\documentclass{article}
\usepackage[latin1]{inputenc} 
\usepackage[T1]{fontenc} 
\usepackage[english]{babel} 
\usepackage{indentfirst} 
\usepackage{latexsym}
\usepackage{amssymb}
\usepackage{amsmath}
\usepackage{amsthm}
\usepackage{url}
\usepackage{enumitem}	
\usepackage{longtable}

\newtheorem{thm}{Theorem}[section]
\newtheorem{cor}[thm]{Corollary}
\newtheorem{lem}[thm]{Lemma}
\newtheorem{prop}[thm]{Proposition}
\newtheorem{defn}[thm]{Definition}

\newtheorem{oss}[thm]{Observation}

\newcommand{\U}{\mathcal{U}}
\newcommand{\Par}{\mathcal{P}}
\newcommand{\F}{\mathcal{F}}
\newcommand{\A}{\mathcal{A}}

\newcommand{\N}{\mathbb{N}}
\newcommand{\Z}{\mathbb{Z}}

\newcommand{\V}{\mathcal{V}}
\newcommand{\W}{\mathcal{W}}
\newcommand{\bN}{\beta\mathbb{N}}

\newcommand{\fe}{\leq_{\F}}

\begin{document}

\title{$\F$-finite embeddabilities of sets and ultrafilters}

\author{Lorenzo Luperi Baglini\thanks{University of Vienna, Faculty of Mathematics, Oskar-Morgenstern-Platz 1, 1090 Vienna, AUSTRIA, e-mail: \texttt{lorenzo.luperi.baglini@univie.ac.at}, supported by grant P25311-N25 of the Austrian Science Fund FWF.}}
\maketitle

\begin{abstract}

Let $S$ be a semigroup, let $n\in\N$ be a positive natural number, let $A,B\subseteq S$, let $\U,\V\in\beta S$ and let let $\F\subseteq\{f:S^{n}\rightarrow S\}$. We say that $A$ is $\F$-finitely embeddable in $B$ if for every finite set $F\subseteq A$ there is a function $f\in\F$ such that $f\left(A^{n}\right)\subseteq B$, and we say that $\U$ is $\F$-finitely embeddable in $\V$ if for every set $B\in\V$ there is a set $A\in\U$ such that $A$ is $\F$-finitely embeddable in $B$. We show that $\F$-finite embeddabilities can be used to study certain combinatorial properties of sets and ultrafilters related with finite structures. We introduce the notions of set and of ultrafilter maximal for $\F$-finite embeddability, whose existence is proved under very mild assumptions. Different choices of $\F$ can be used to characterize many combinatorially interesting sets/ultrafilters as maximal sets/ultrafilters, for example thick sets, AP-rich sets, $\overline{K(\beta S)}$ and so on. The set of maximal ultrafilters for $\F$-finite embeddability can be characterized algebraically in terms of $\F$. This property can be used to give an algebraic characterization of certain interesting sets of ultrafilters, such as the ultrafilters whose elements contain, respectively, arbitrarily long arithmetic, geoarithmetic or polynomial progressions. As a consequence of the connection between sets and ultrafilters maximal for $\F$-finite embeddability we are able to prove a general result that entails, for example, that given a finite partition of a set that contains arbitrarily long geoarithmetic (resp. polynomial) progressions, one cell must contain arbitrarily long geoarithmetic (resp. polynomial) progressions. Finally we apply $\F$-finite embeddabilities to study a few properties of homogeneous partition regular diophantine equations. Some of our results are based on connections between ultrafilters and nonstandard models of arithmetic.

\end{abstract}

\section{Introduction}
The notion of finite embeddability of sets of integers was introduced by M. Di Nasso in \cite{8} to study problems related to difference sets in combinatorial number theory (see also \cite{22} where this notion was implicitly used by I. Z. Ruzsa). 
\begin{defn}[\cite{8}, \S 4]\label{gj} Let $A,B\subseteq\Z$. We say that $A$ is finitely embeddable in $B$ if each finite subset $F$ of $A$ has a rightward translate $F + k=\{f+k\mid f\in F\}$ included in $B$.\end{defn} 
In \cite{4} A. Blass and M. Di Nasso considered the notion of finite embeddability of sets of natural numbers, as well as a related notion defined for ultrafilters.
\begin{defn}[\cite{4}, \S 1] Let $\U,\V\in\bN$. We say that $\U$ is finitely embeddable in $\V$ if for every set $B\in\V$ there is a set $A\in\U$ such that $A$ is finitely embeddable in $B$. \end{defn}
A. Blass and M. Di Nasso proved many basic properties of the finite embeddability of sets and ultrafilters on $\N$, for example that the finite embeddability of sets is related with the notion of "leftward $\V$-shift" (a notion first considered by P. Krautzberger in \cite{13} and independently introduced by M. Beiglb\"{o}ck in \cite{1} to give a ultrafilter proof of Jin's theorem). They also showed how finite embeddability can be characterized in terms of nonstandard extensions of $\N$. In \cite{16} we continued the study of the finite embeddability, both with standard and nonstandard methods. Our main result was that there exist ultrafilters maximal for finite embeddability and that the set of such maximal ultrafilters is the closure of the minimal bilateral ideal of $(\bN,\oplus)$, namely $\overline{K(\bN,\oplus)}$. We also provided some applications to combinatorial number theory.\par
The main idea behind the applications of the finite embeddability is that if $A$ is finitely embeddable in $B$ then for every "combinatorially interesting" finite structure in $A$ there is a translate of that structure in $B$. This allows to deduce some combinatorial properties of $B$. In this paper we want to generalize this idea by modifying the notion of "finite embeddability" that we are going to consider, by allowing the use of more general functions to embed finite subsets of a set $A$ into a set $B$. In this way we will be able to threat in a single setting different notions such as arithmetic, polynomial and geoarithmetic progressions, partition regular diophantine homogeneous equations, piecewise syndetic sets and so on (see Section \ref{appl}). Even if we are mostly interested in applications to combinatorial number theory on $\N$, $\F$-finite embeddabilities can be introduced for a generic semigroup\footnote{Namely, we assume to have a binary operation $\cdot$ defined on $S$ such that $(S,\cdot)$ is a semigroup. For all properties and definitions regarding semigroups we refer to \cite{11}.} $S$.

\begin{defn} Let $S$ be a semigroup, let $A,B\subseteq S$, let $n\geq 1$ be a natural number and let $\F\subseteq\left\{f:S^{n}\rightarrow S\right\}$. We say that $A$ is $\F$-finitely embeddable in $B$ and we write $A\leq_{\F} B$ if for every finite set $F\subseteq A$ there is $f\in\F$ such that
\begin{equation*} f\left(A^{n}\right):=\left\{f(a_{1},\dots,a_{n})\mid \forall i\leq n \ a_{i}\in A\right\}\subseteq B. \end{equation*}
\end{defn}

\begin{defn} Let $\U,\V\in \beta S$, let $n\geq 1$ and let $\F\subseteq\left\{f:S^{n}\rightarrow S\right\}$. We say that $\U$ is $\F$-finitely embeddable in $\V$ and we write $\U\leq_{\F} \V$ if for every set $B\in\V$ there exists $A\in\U$ such that $A\leq_{\F} B$. \end{defn}

In Section \ref{ciu} we study some basic properties of $\F$-finite embeddabilities for semigroups, and in Section \ref{cui} we study some basic properties of $\F$-finite embeddabilities of ultrafilters and we introduce the notions of ultrafilter maximal (or weakly maximal) for $\F$-finite embeddability. In Section \ref{uci} we study such ultrafilters by means of a nonstandard characterization of $\F$-finite embeddabilities of ultrafilters based on certain particular iterated nonstandard models of arithmetics (nonstandard techniques are used only Section \ref{uci} and to prove Lemma \ref{lola}). We use this characterization to provide an algebraic characterization of maximal ultrafilters for a large family of $\F$-finite embeddabilities, which is used to improve and generalize the main results obtained in \cite{4}, \cite{16}. In Section \ref{uic} we prove that the families of sets maximal for $\F$-finite embeddabilities are strongly partition regular whenever they are partition regular (see Definition \ref{vulo}). Finally, in Section \ref{appl} we show some examples of applications of $\F$-finite embeddabilities. In particular, we obtain an algebraic characterization of the sets
\begin{itemize}
	\item $\left\{\U\in\bN\mid\forall A\in\U \ A \ \text{is AP-rich}\right\}$;
	\item $\left\{\U\in\bN\mid\forall A\in\U \ A \ \text{contains arbitrarily long geoarithmetic progressions}\right\}$;
	\item $\left\{\U\in\bN\mid\forall A\in\U \ A \ \text{contains arbitrarily long polynomial progressions}\right\}$,
\end{itemize}
and we also prove that the family of sets that contain arbitrarily long geoarithmetic (resp. polynomial) progressions is strongly partition regular. Finally, we apply $\F$-finite embeddabilities to show some properties of partition regular homogeneous diophantine equations.

\section{$\F$-finite embeddabilities of semigroups}\label{ciu}

In this section we want to prove a few basic properties of $\F$-finite embeddabilities between subsets of semigroups. We start our study by fixing some notations that will be used throughout the paper.
\begin{defn} Let $n\geq 1$ be a natural number, and let $S$ be a semigroup. We let:
\begin{itemize}
  \item $\Par(S)=\{A\subseteq S\}$;
	\item $\Par_{fin}(S)=\{A\subseteq S\mid A$ is finite$\}$;
	\item $\mathfrak{F}\left(S^{n},S\right)=\left\{f:S^{n}\rightarrow S\right\}$;
	\item $\mathcal{F}^{R}_{\left(S,\cdot\right)}=\left\{f_{r}\in \mathfrak{F}(S,S)\mid r\in R \ \text{and} \ \forall s\in S \ f_{r}\left(s\right)=s\cdot r\right\}$;
	\item $\mathcal{F}^{L}_{\left(S,\cdot\right)}=\left\{g_{r}\in \mathfrak{F}(S,S)\mid r\in R \ \text{and} \ \forall s\in S \ g_{r}\left(s\right)=r\cdot s\right\}$.
\end{itemize}
When $S$ is commutative, we set $\mathcal{F}_{\left(S,\cdot\right)}:=\mathcal{F}^{R}_{\left(S,\cdot\right)}=\mathcal{F}^{L}_{\left(S,\cdot\right)}$.
\end{defn}
In the next proposition we summarize the first basic propertis of $\F$-finite embeddabilities of sets.

\begin{prop}\label{listona} Let $S$ be a semigroup, let $A,A_{1},A_{2},B,B_{1},B_{2}\in\Par\left(S\right)$ and let $\mathcal{F},\mathcal{F}_{1},\dots,\mathcal{F}_{k}\in \mathfrak{F}(S^{n},S)$. Then:
\begin{enumerate}[leftmargin=*,label=(\roman*),align=left ]
	\item if $\mathcal{F}=\left\{f\right\}$ then $A\leq_{\left\{f\right\}} B$ if and only if $f\left(A^{n}\right)\subseteq B$;
	\item\label{due} if $A\leq_{\mathcal{F}_{1}\cup....\cup\mathcal{F}_{k}} B$ then there is an index $i\leq k$ such that $A\leq_{\mathcal{F}_{i}} B$;
	\item if $A_{1}\subseteq A_{2}$ and $A_{2}\leq_{\mathcal{F}} B$ then $A_{1}\leq_{\mathcal{F}} B$;
	\item if $B_{1}\subseteq B_{2}$ and $A\leq_{\mathcal{F}} B_{1}$ then $A\leq_{\mathcal{F}} B_{2}$.
\end{enumerate}

\end{prop}

\begin{proof} The only proof which is not trivial is that of \ref{due}. It is clear that it is sufficient to show that this property holds for $k=2$, since the general case follows easily by induction. Let us suppose that $A\leq_{\mathcal{F}_{1}\cup\mathcal{F}_{2}} B$. If $F\leq_{\mathcal{F}_{1}} B$ for every finite subset $F$ of $A$ then, by definition, $A\leq_{\mathcal{F}_{1}} B$. Otherwise there exists a finite subset $F_{0}$ of $A$ such that 
\begin{equation}\label{zorro}f\left(F_{0}^{n}\right)\nsubseteq B \ \mbox{for every} \ f\in\mathcal{F}_{1}.\end{equation}
Now let $F\in\Par_{fin}(A)$. Since $A\leq_{\mathcal{F}_{1}\cup\mathcal{F}_{2}} B$ there exists $f\in\mathcal{F}_{1}\cup\mathcal{F}_{2}$ such that $f\left([F_{0}\cup F]^{n}\right)\subseteq B$. By (\ref{zorro}) we have that $f\in\mathcal{F}_{2}$ so, in particular, $F\leq_{\mathcal{F}_{2}} B$. Since this holds for every finite subset $F$ of $A$ we have that $A\leq_{\mathcal{F}_{2}} B$.\end{proof}

It is not difficult to notice (see e.g. \cite{4}, \cite{14}) that $\left(\Par\left(\N\right),\leq_{\F_{\left(\N,+\right)}}\right)$ is a preorder, namely that the finite embeddability is reflexive and transitive on $\Par\left(\N\right)$. This will not be always the case for a general semigroup $S$ and a general family $\F$ of functions; however, it is not difficult to isolate the conditions that ensure that $\leq_{\F}$ is a preorder.

\begin{prop}\label{preorder} Let $S$ be a semigroup. The relation $\leq_{\mathcal{F}}$ is 
\begin{enumerate}[leftmargin=*,label=(\roman*),align=left ]
	\item\label{uno.2} transitive if and only if for every finite $F\subseteq S$, for every functions $f,g$ in $\mathcal{F}$ there is a function $h$ in $\mathcal{F}$ such that $h\left(F^{n}\right)\subseteq g\left([f\left(F^{n}\right)]^{n}\right)$;
	\item\label{due.2} reflexive if and only if for every finite $F\subseteq S$ there is a function $f$ in $\mathcal{F}$ such that $f\left(F^{n}\right)\subseteq F$. 
\end{enumerate}
 \end{prop}

\begin{proof} \ref{uno.2} Let $\leq_{\mathcal{F}}$ be transitive. Let $F$ be a finite subset of $S$ and let $f,g$ be functions in $\mathcal{F}$. Then $F\leq_{\mathcal{F}} f\left(F^{n}\right)$ and $f\left(F^{n}\right)\leq_{\mathcal{F}} g\left([f\left(F^{n}\right)]^{n}\right)$ so $F\fe g\left([f\left(F^{n}\right)]^{n}\right)$. As $F$ is finite, this happens if and only if there is a function $h$ in $\mathcal{F}$ such that $h\left(F^{n}\right)\subseteq g\left([f\left(F^{n}\right)]^{n}\right)$.\par
Conversely, let $A,B,C$ be subsets of $S$ such that $A\leq_{\mathcal{F}} B$ and $B\leq_{\mathcal{F}} C$. Let $F$ be a finite subset of $A$ and let $f,g\in\F$ be such that $f\left(F^{n}\right)\subseteq B$ and $g\left([f\left(F^{n}\right)]^{n}\right)\subseteq C$. If $h\in\F$ is such that $h\left(F^{n}\right)\subseteq g\left([f\left(F^{n}\right)]^{n}\right)$, we have that $h\left(F^{n}\right)\subseteq C$, so $A\fe C$ and $\F$ is transitive.\par
\ref{due.2} This equivalence is a direct consequence of the definitions.\end{proof}

\begin{defn} We say that a family of functions $\mathcal{F}\subseteq \mathfrak{F}(S^{n},S)$ is transitive (resp. reflexive) if $\left(\Par\left(S\right),\fe\right)$ is transitive (resp. reflexive).\end{defn}

In particular, let us note that, for every semigroup $S$, both $\mathcal{F}^{R}_{\left(S,\cdot\right)}$ and $\mathcal{F}^{L}_{\left(S,\cdot\right)}$ are transitive sets of functions, and that if $S$ has an identity then they are also reflexive. Moreover, from Proposition \ref{preorder} we deduce that if $n=1$ then $g\left([f\left(F^{n}\right)]^{n}\right)=\left(g\circ f\right)\left(F\right)$. In this case an immediate consequence of Proposition \ref{preorder} is that if $\mathcal{F}$ is closed under composition then $\leq_{\mathcal{F}}$ is transitive, and if the identity map $i:S\rightarrow S$ is in $\mathcal{F}$ then $\leq_{\mathcal{F}}$ is reflexive. In particular if $\F\subseteq \mathfrak{F}(S,S)$ and $\left(\F,\circ\right)$ is a monoid then $\left(\Par\left(S\right),\fe\right)$ is a preorder.\par
The basic idea behind the applications of $\F$-finite embeddabilities (see e.g. \cite{4}, \cite{8}, \cite{16} for some examples regarding $\leq_{\F_{(\N,+)}}$) is that the relation $A\leq_{\mathcal{F}} B$ permits to "lift" certain combinatorial properties of $A$ to $B$. The idea is that sets that are maximal with respect to $\F$-finite embeddabilities are combinatorially rich (in a sense that has to be precised, and that will depend on $\F$). Since in the following we will study in detail these sets, let us fix some notations. 

\begin{defn} Let $S$ be a semigroup and let $\mathcal{F}\subseteq \mathfrak{F}(S,S)$ be transitive. We denote by $M\left(S,\mathcal{F}\right)$ the set of maximal elements in $\left(\Par\left(S\right),\fe\right)$, namely
\begin{equation*} M\left(S,\mathcal{F}\right)=\left\{A\subseteq S\mid \ S\leq_{\F} A\right\}. \end{equation*}
\end{defn}

Clearly $S\in M\left(S,\mathcal{F}\right)$, so $M\left(S,\mathcal{F}\right)\neq\emptyset$. A simple characterization of the sets in $M\left(S,\mathcal{F}\right)$ is stated in the following proposition, whose proof follows easily from the definitions.

\begin{prop}\label{maxset} Let $A\subseteq S$. The following conditions are equivalent:
\begin{enumerate}[leftmargin=*,label=(\roman*),align=left ]
	\item $A\in M\left(S,\mathcal{F}\right)$;
	\item for every finite set $F\subseteq S$ we have that $F\leq_{\mathcal{F}} A$.
\end{enumerate}
\end{prop}

{\bfseries Example.} Let us give some examples.
\begin{enumerate}[leftmargin=*,label=(\roman*),align=left ]
  \item We recall that (see e.g. \cite{11}, Definition 4.45) a subset $A$ of a semigroup $S$ is thick iff for every finite $F\subseteq S$ there is $s\in S$ such that $F\cdot s\subseteq A$. Therefore, from Proposition \ref{maxset} it is immediate to deduce that, for every semigroup $S$, 
\begin{equation*} M\left(S,\mathcal{F}^{R}_{\left(S,\cdot\right)}\right)=\left\{A\subseteq S\mid A \ \text{is thick}\right\}. \end{equation*}
	\item\label{aritm} Let $S=\N$ and let
	\begin{equation*} \mathcal{A}=\left\{f_{a,b}\left(x\right)\in\mathfrak{F}(\N,\N)\mid a,b\in\N, \ b>0 \ \text{and} \ f_{a,b}\left(x\right)=a+bx \ \forall x\in\N\right\}. \end{equation*}
	Then from Proposition \ref{maxset} we deduce that 	
	\begin{equation*} M\left(\N,\mathcal{A}\right)=\left\{A\subseteq\N\mid A \ \text{is AP-rich}\right\}\end{equation*}
	since, given any $k\in\N$, for every $f_{a,b}\in\A$ we have that $f_{a,b}\left([0,\dots,k]\right)$ is an arithmetic progression of length $k+1$.
	\item Let again $S=\N$ and let 
	\begin{multline*} \mathcal{G}=\Big\{f_{r,a,b}\left(n,m\right)\in\mathfrak{F}(\N,\N)\mid r,a,b\in\N, r>1, b>0 \ \text{and} \\
	f_{r,a,b}\left(n,m\right)=r^{n}\left(a+mb\right) \ \forall \left(n,m\right)\in\N^{2}\Big\}. \end{multline*}	
	Let us call GAP-rich a set $A\subseteq\N$ if it contains arbitrarily long geoarithmetic progressions, namely if for every $k\in\N$ there are $r>1, b>0, a\in\N$ such that $r^{i}(a+jb)\in A$ for every $1\leq i,j\leq k$. Then, similarly to the case \ref{aritm}, by applying Proposition \ref{maxset} we deduce that 	
	\begin{equation*} M\left(\N,\mathcal{G}\right)=\left\{A\subseteq\N\mid A \ \text{is GAP-rich}\right\}.\end{equation*}
	\item Let $\Sigma=\{a,b\}$, let $\Sigma^{+}$ be the free semigroup on $\Sigma$ and let 
	\begin{equation*} \mathcal{F}=\left\{f_{n}\in\mathfrak{F}(\Sigma^{+},\Sigma^{+})\mid n\in\N, f_{n}\left(w\right)=wa^{n}\ \forall w\in\Sigma^{+}\right\}.\end{equation*}
	Then 
	\begin{multline*} M\left(\Sigma^{+},\mathcal{F}\right)=\Big\{A\subseteq\Sigma^{+}\mid \forall m\in\N,\ \forall w_{1},\dots,w_{m}\in\Sigma^{+} \\
	\exists n\in\N \ \text{such that} \ w_{i}a^{n}\in A \ \forall i\leq m\Big\}.\end{multline*}
\end{enumerate}

In Section \ref{uic} we will prove a result regarding sets maximal for $\F$-finite embeddabilities that will be useful for applications. We conclude this section by proving a relationship between $\F^{R}_{(S,\cdot)}$-finite embeddabilities and a general notion of density for semigroups. In \cite{4} the authors observed that Banach density is increasing with respect to finite embeddability, namely that, for every $A,B\subseteq\N$, if $A\leq_{\F_{\left(\N,+\right)}} B$ then $BD\left(A\right)\leq BD\left(B\right)$. Our generalization of this result to arbitrary semigroups is based on a general notion of density for semigroups introduced by N. Hindman and D. Strauss in \cite{12}, that extends the usual notions of density for left amenable semigroups based on nets of finite sets. Let us recall its definition.

\begin{defn}\label{GenBan} Let $\left(D,\leq\right)$ be an upward directed set, let $S$ be a semigroup, let $F=\langle F_{n} \rangle_{n\in D}$ be a net\footnote{Namely $F_{n}\subseteq F_{m}$ whenever $n\leq m$.} in $\Par_{fin}\left(S\right)$ and let $A\subseteq S$. Then 
\begin{multline*} d^{\ast}_{F}\left(A\right)=\sup\{\alpha\in\mathbb{R}\mid \left(\forall m\in D\right)\left(\exists n\geq m\right) \\\left(\exists x\in S\cup\{1\}\right) \left(|A\cap (F_{n}\cdot x)|\geq \alpha |F_{n}|\right)\}, \end{multline*}
where we have set $|A\cap (F_{n}\cdot 1)|=|A\cap F_{n}|$ (namely, $s\cdot 1=s$ for every $s\in S$).
\end{defn}

 We also recall that, given a natural number $b$, the semigroup $S$ is $b$-weakly left cancellative if, for all $x,y\in S$, we have that 
\begin{equation*} |\left\{s\in S\mid s\cdot x=y\right\}|\leq b. \end{equation*}
\begin{thm}\label{UppDens} Let $S$ be a $b$-weakly left cancellative semigroup and let $A,B\subseteq S$. If $A\leq_{\mathcal{F}^{R}_{\left(S,\cdot\right)}} B$ then $\frac{1}{b}d^{\ast}_{F}\left(A\right)\leq  d^{\ast}_{F}\left(B\right)$. \end{thm}

\begin{proof} Let $\alpha\in [0,1]\subseteq\mathbb{R}$ and let us suppose that $\alpha\leq d^{\ast}_{F}\left(A\right)$. Let $m\in D$. Let $n\geq m$ and $x\in S\cup\{1\}$ be such that $|A\cap (F_{n}\cdot x)|\geq\alpha |F_{n}|$. $A\cap (F_{n} \cdot x)$ is a finite subset of $A$, therefore there exists $y\in S$ such that $\left(A\cap (F_{n} \cdot x)\right)\cdot y\subseteq B$. Hence $|B\cap (F_{n} \cdot x\cdot y)|\geq |\left(A\cap (F_{n}\cdot x)\right)\cdot y|$. Since $S$ is a $b$-weakly left cancellative semigroup, we have that $|\left(A\cap (F_{n}\cdot x)\right)\cdot y|\geq\frac{1}{b}|A\cap (F_{n}\cdot x)|\geq\frac{1}{b}\alpha$. Since this holds for every $\alpha\leq d^{\ast}_{F}\left(A\right),$ we deduce that $d^{\ast}_{F}\left(B\right)\geq \frac{1}{b}d^{\ast}_{F}\left(A\right)$. \end{proof}

Two immediate consequences of Theorem \ref{UppDens} are the following.

\begin{cor} Let $S$ be a $b$-weakly left cancellative semigroup and let $T$ be a thick subset of $S$. Then $d^{\ast}_{F}\left(T\right)\geq \frac{1}{b}$. \end{cor}

\begin{proof} If $T$ is thick then $S\leq_{\mathcal{F}_{\left(S,\cdot\right)}} T$. The thesis follows by observing that, clearly, $d^{\ast}_{F}\left(S\right)=1$. \end{proof}

\begin{cor}\label{zavagno} Let $S$ be a $b$-weakly left cancellative semigroup. Let $A,B\in\Par\left(S\right)$. If $d^{\ast}_{F}\left(A\right)>0$ and $A\leq_{\mathcal{F}^{R}_{\left(S,\cdot\right)}} B$ then $d^{\ast}_{F}\left(B\right)>0$. \end{cor}

Let us note that if $\left(S,\cdot\right)=\left(\N,+\right)$ and $F_{n}=\left\{1,...,n\right\}$ for every $n\in\mathbb{N}$ then $d^{\ast}_{F}$ is the upper Banach density, so the result regarding Banach density and finite embeddability on $\N$ is a particular case of Corollary \ref{zavagno}.

\section{$\F$-finite embeddabilities of ultrafilters}\label{cui}

In this section we study some basic properties of $\left(\beta S,\fe\right)$. Let us observe that, if $\left(\Par\left( S\right),\leq_{\mathcal{F}}\right)$ is transitive, then $\left(\beta S,\fe\right)$ is transitive: in fact, let $\U\fe\V\fe\mathcal{W}$. Let $A\in\mathcal{W}$ and let $B\in\V, C\in\U$ be such that $B\fe A, C\fe B$. Since $\left(\Par\left( S\right),\leq_{\mathcal{F}}\right)$ is transitive we have that $C\fe A$, therefore $\U\fe\W$ and the transitive property of $\left(\beta S,\fe\right)$ is proved. It is immediate to observe that an analogous result holds when $\F$ is reflexive.\par
From now on we assume $\left(\beta S,\fe\right)$ to be transitive. We fix some notations that will be important throughout the paper. 
\begin{defn} Let $S$ be a semigroup. For every ultrafilter $\U\in\beta S$ we set 
\begin{equation*} \U^{\left(n\right)}:=\overbrace{\U\times\cdots\times\U}^{n \ \text{times}}=\left\{A_{1}\times\dots\times A_{n}\mid A_{i}\in\U \ \forall i\leq n\right\} \end{equation*}
and
\begin{equation*} F\left(\U^{\left(n\right)}\right)=\left\{A\subseteq S^{n}\mid \exists B\in\U^{\left(n\right)} \ \text{such that} \ B\subseteq A\right\}. \end{equation*}
Moreover we let $\mathfrak{G}\left(\U^{\left(n\right)}\right)=\left\{\W\in\beta\left(\N^{n}\right)\mid F\left(\U^{\left(n\right)}\right)\subseteq\W\right\}$.\end{defn}
Given a function $f:S^{n}\rightarrow S$ we will denote by $\overline{f}:\beta (S^{n})\rightarrow \beta S$ the unique continuous extension of $f$. Let us recall (see e.g. \cite{11}, Lemma 3.30) that $\overline{f}$ is defined as follows:
\begin{equation*} \forall \U\in\beta (S^{n}) \ \overline{f}(\U)=\left\{A\subseteq S\mid f^{-1}(A)\in\U\right\}.\end{equation*}
Moreover we will denote by $\odot$ the extension of the semigroup operation $\cdot$ to $\beta S$ defined as follows (see e.g. \cite{11}, Section 4.1 for the properties of this extension):
\begin{equation*} \forall \U,\V\in\beta S \ \U\odot\V=\left\{A\subseteq S\mid \left\{s\in S\mid \left\{r\in S\mid s\cdot r\in A\right\}\in\V\right\}\in\U\right\}. \end{equation*} 
As usual, we will also identify every element $s\in S$ with the principal ultrafilter
\begin{equation*} \U_{s}=\{A\subseteq S\mid s\in A\}. \end{equation*}
We start our study of $\left(\beta S,\fe\right)$ with the analogous of Proposition \ref{listona} for ultrafilters.

\begin{prop}\label{listona2} Let $S$ be a semigroup, let $n\geq 1$ be a natural number, let $\U,\V$ be ultrafilters on $S$ and let $\F,\F_{1},\F_{2}\subseteq\mathfrak{F}(S^{n},S)$. Then:

\begin{enumerate}[leftmargin=*,label=(\roman*),align=left ]
\item\label{list1} if $\mathcal{F}=\left\{f\right\}$ then $\U\fe\V$ if and only if $\V=\overline{f}\left(\W\right)$ for every $\W\in \mathfrak{G}\left(\U^{\left(n\right)}\right)$;
\item\label{list2} $\U\leq_{\mathcal{F}_{1}\cup\mathcal{F}_{2}}\V$ if and only if $\U\leq_{\mathcal{F}_{1}}\V$ or $\U\leq_{\mathcal{F}_{2}}\V$;
\item\label{list3} if $\mathcal{F}_{1}\subseteq \mathcal{F}_{2}$ and $\U\leq_{\mathcal{F}_{1}}\V$ then $\U\leq_{\mathcal{F}_{2}}\V$.
\end{enumerate}

\end{prop}

\begin{proof} \ref{list1} From Proposition \ref{listona} we get the following equivalence
\begin{equation*} \U\leq_{\left\{f\right\}}\V\Leftrightarrow \forall A\in\V \ \exists B\in\U  \ f\left(B^{n}\right)\subseteq A \Leftrightarrow 
\forall A\in \V  \ f^{-1}\left(A\right)\in F\left(\U^{\left(n\right)}\right)\end{equation*}
and the thesis follows since $\W\in\mathfrak{G}\left(\U^{(n)}\right)\Leftrightarrow F\left(\U^{(n)}\right)\subseteq \W$.\par
\ref{list2} Let us suppose that $\U\leq_{\mathcal{F}_{1}\cup \mathcal{F}_{2}}\V$. There are only two possibilities:
\begin{enumerate}
	\item For every set $B$ in $\V$ there is a set $A$ in $\U$ such that $A\leq_{\mathcal{F}_{1}} B$;
	\item There is a set $B$ in $\V$ such that, for every set $A$ in $\U$, $A$ is not $\mathcal{F}_{1}$-finitely embeddable in $B$.
\end{enumerate}
In the first case $\U\leq_{\mathcal{F}_{1}}\V$; in the second case $\U\leq_{\mathcal{F}_{2}}\V$. In fact, let $A\in\V$. $A\cap B\in\V$, so there exists $C\in\U$ such that $C\leq_{\F_{1}\cup \F_{2}} A\cap B$. But $C \nleq_{\F_{1}} A\cap B$ so from Proposition \ref{listona} it follows that $C\leq_{\F_{2}} A\cap B$. In particular $C\leq_{\F_{2}} A$, so $\U\leq_{\F_{2}}\V$ as claimed.\par
\ref{list3} This is a trivial consequence of the definitions.\end{proof}

When on $S$ is defined an order relation $\leq$, a natural question that arises is if $\left(\beta S,\fe\right)$ is an extension of $\left(S,\leq\right)$ having $\left(S,\leq\right)$ as its initial segment, namely if 
\begin{itemize}
	\item $\fe$ coincides with $\leq$ when restricted to $S$;
	\item every principal ultrafilter is $\F$-finitely embeddable in every nonprincipal ultrafilter;
	\item nonprincipal ultrafilters are not $\F$-finitely embeddable in principal ultrafilters.
\end{itemize}
When this happens we say that $\left(\beta S,\mathcal{F}\right)$ is a coherent extension of $\left(S,\leq\right)$. In \cite{14}, \cite{16} we proved that $\left(\bN,\F_{\left(\N,+\right)}\right)$ is a coherent extension of $\left(\N,\leq\right)$. For a general ordered semigroup and a general family $\mathcal{F}$ this property does not hold, as can be seen by considering the semigroup $\left(\N,+\right)$ with its usual ordering and
\begin{equation*} \mathcal{F}=\left\{ c_{n}\mid n\in \N \right\}\cup\left\{id\right\}\end{equation*} 
where, for every $n\in\N$, $c_{n}$ is the constant function with value $n$ and $id$ is the identity function on $\N$. It is immediate to see that $\U\fe\V$ for every $\U,\V\in \bN$, so $\left(\bN,\F\right)$ is not a coherent extension of $(\N,\leq)$.\par
To characterize coherent extensions we use two lemmas, whose proofs follow easily from the definitions.

\begin{lem}\label{Lem1} Let $S$ be a semigroup, let $n\geq 1$ be a natural number and let $\F\subseteq \mathfrak{F}(S^{n},S)$. Let $s,r\in S$. The following two conditions are equivalent:
\begin{enumerate}[leftmargin=*,label=(\roman*),align=left ]
	\item $\U_{s}\fe \U_{r}$;
	\item there is a function $f$ in $\mathcal{F}$ such that $f\left(r,r,\dots,r\right)=s$.
\end{enumerate}
\end{lem}

\begin{lem}\label{Lem2} Let $S$ be a semigroup, let $n\geq 1$ be a natural number and let $\F\subseteq \mathfrak{F}(S^{n},S)$. The following two conditions are equivalent:
\begin{enumerate}[leftmargin=*,label=(\roman*),align=left ]
	\item every principal ultrafilter $\U$ is $\mathcal{F}$-finitely embeddable in every nonprincipal ultrafilter $\V$;
	\item for every element $s\in S$, for every infinite subset $A$ of $S$ there is a function $f$ in $\mathcal{F}$ such that $f\left(s,s,\dots,s\right)\in A$.
\end{enumerate} 
\end{lem}

By combining Lemmas \ref{Lem1} and \ref{Lem2} we obtain the desired characterization of coherent extension.

\begin{prop}\label{initseg} $\left(\beta S,\mathcal{F}\right)$ is a coherent extension of $\left(S,\leq\right)$ if and only if $\F$ satisfies the following two conditions:
\begin{enumerate}[leftmargin=*,label=(\roman*),align=left ]
	\item $\forall f\in\F, \forall r,s\in S$ if $r<s$ then $f\left(r,r,\dots,r\right)<f\left(s,s,\dots,s\right)$;
	\item for every $s<r\in S$ there exists $f\in\F$ such that $f\left(s,s,\dots,s\right)=r$.
\end{enumerate}
\end{prop}

From Proposition \ref{initseg} it follows that $\left(\bN,\leq_{\F_{\left(\N,+\right)}}\right)$ and $\left(\bN,\mathcal{A}\right)$ are coherent extensions of $(\N,\leq)$ while $\left(\bN,\leq_{\F_{\left(\N,\cdot\right)}}\right), \left(\bN,\mathcal{G}\right)$ are not.\par
We now want to study "maximal" ultrafilters. Since, in general, $\left(\beta S,\fe\right)$ is not an order, we precise what we mean by "maximal ultrafilter" in the following definition.

\begin{defn} Let $\F$ be a transitive set of functions and let $\U\in\beta S$. Then we say that $\U$ is a maximal ultrafilter if, for every $\V\in\beta S$, $\V\fe\U$. We say that $\U$ is weakly maximal if, for every $\V\in\beta S$, if $\U\leq\V$ then $\V\leq\U$. \end{defn}

Notice that every maximal ultrafilter is weakly maximal by definition and that, since we assume $\F$ to be transitive, if there exists a maximal ultrafilter then every weakly maximal ultrafilter is maximal. We will denote by $W\left(\beta S,\F\right)$ the set of weakly maximal ultrafilters in $\left(\beta S,\leq_{\F}\right)$, and by $M\left(\beta S,\F\right)$ the set of maximal ultrafilters in $\left(\beta S,\leq_{\F}\right)$.\par
In \cite{16} we showed that in $\left(\bN,\leq_{\F_{\left(\N,+\right)}}\right)$ there are maximal ultrafilters, and that the set of such maximal ultrafilters is $\overline{K\left(\bN,\oplus\right)}$, namely the closure of the minimal bilateral ideal of $\left(\bN,\oplus\right)$. This result was used to (re)prove some combinatorial properties of the ultrafilters in $\overline{K\left(\bN,\oplus\right)}$, which can be used to deduce some combinatorial properties of piecewise syndetic sets in $\left(\N,+\right)$. Our main aim now is to find the properties that ensure the existence of (weakly) maximal ultrafilters for a generic pair $\left(\beta S,\mathcal{F}\right)$.\par
To perform our study we will use the following terminology: chains with respect to $\fe$ will be called $\fe$-chains and their upper bounds (when they exist) will be called $\fe$-upper bounds.\par
Our claim is that every $\fe$-chain has a $\fe$-upper bound. To prove this claim we recall two results which have been proved, e.g., in \cite{16}.

\begin{lem}\label{ultraordine} If $I$ is a totally ordered set then there is an ultrafilter $\V$ on $I$ such that, for every element $i\in I$, the set 

\begin{center} $G_{i}=\left\{j\in I\mid j\geq i\right\}$.\end{center}

is included in $\V$.

\end{lem} 

\begin{prop}\label{ultraordine2} Let $I$ be a totally ordered set and let $\V$ be given as in Lemma \ref{ultraordine}. Then for every $A\in \V$, $i\in I$ there exists $j\in A$ such that $i\leq j$. \end{prop}

By applying Lemmas \ref{ultraordine} and \ref{ultraordine2} it is possible to show that every $\fe$-chain has a $\fe$-upper bound.

\begin{prop}\label{guru} Let $S$ be a semigroup, let $n\geq 1$ be a natural number and let $\mathcal{F}\subseteq \mathfrak{F}(S^{n},S)$ be a transitive family. Then every $\fe$-chain $\langle \U_{i}\mid i\in I \rangle$ in $\beta S$ has a $\fe$-upper bound $\U$. \end{prop}

\begin{proof} If $I$ has a greatest element $i$ then the ultrafilter $\U_{i}$ is the $\fe$-upper bound of the $\fe$-chain.\par
Otherwise, let us suppose that $I$ has not a greatest element. Let $\V$ be an ultrafilter on $I$ with the property expressed in Lemma \ref{ultraordine}. Let us consider the limit ultrafilter\footnote{We recall that $\V-\lim\limits_{i\in I}\U_{i}$ is the ultrafilter on $S$ defined by the following relation: for every $A\subseteq S$, $A\in\V-\lim\limits_{i\in I}\U_{i}\Leftrightarrow\left\{i\in I\mid A\in\U_{i}\right\}\in\V$.} 

\begin{center} $\U=\V-\lim_{i\in I}\U_{i}$.\end{center}

We claim that $\U$ is the $\fe$-upper bound of the given $\fe$-chain.\par
Let $B\in\U$ and $i\in I$. Since $B\in\U$, we have that

\begin{center} $I_{B}=\left\{i\in I\mid B\in\U_{i}\right\}\in\V,$ \end{center}

so $I_{B}\cap G_{i} \neq\emptyset$. Let $j\in I_{B}\cap G_{i}$. Since $B\in\U_{j}$, and $\U_{i}\fe\U_{j}$, there is a set $A\in\U_{i}$ such that $A\leq_{\mathcal{F}} B$; this proves that $\U_{i}\fe\U$ for every index $i\in I$, so $\U$ is a $\fe$-upper bound for the chain.\end{proof}

\begin{cor}\label{orere} Let $S$ be a semigroup, let $n\geq 1$ be a natural number and let $\F\subseteq\mathfrak{F}(S^{n},S)$ be transitive and reflexive. Then $W\left(\beta S,\F\right)\neq\emptyset$.\end{cor}

\begin{proof} Let $\equiv_{\F}$ be the equivalence relation defined on $\beta S$ by setting
\begin{equation*} \U\equiv_{\F}\V\Leftrightarrow \U\fe\V \ \text{and} \ \V\fe\U. \end{equation*}
For every $\U\in\beta S$ we let $[\U]$ denote its equivalence class with respect to $\equiv_{\F}$, and we let 
\begin{equation*} X=\left\{[\U]\mid\U\in\beta S\right\}.\end{equation*}
It is immediate to observe that if we set, $\forall \U,\V\in\beta S$, 
\begin{equation*} [\U]\leq [\V]\Leftrightarrow \U\fe\V \end{equation*}
then we have that $\left(X,\leq\right)$ is a partially ordered set. Moreover, linear chains in $(X,\leq)$ correspond to $\fe$-chains in $(\beta S,\fe)$, so Proposition \ref{guru} ensures that every chain in $(X,\leq)$ has an upper bound. Therefore, by Zorn's Lemma there is a maximal element $[\U]$ in $(X,\leq)$ so, by construction, $\U\in W(\beta S,\fe)$. \end{proof}

In \cite{16} to prove deduce the existence of maximal elements from the existence of weakly maximal elements in $\left(\bN,\leq_{\F}\right)$ we used that $\left(\bN,\leq_{\F_{(\N,+)}}\right)$ is upward directed\footnote{Let us observe that the converse holds as well: in fact, if there is a maximal element in $\left(\bN,\leq_{\F}\right)$ then $\left(\bN,\leq_{\F}\right)$ is trivially upward directed.}, namely that for every $\U,\V\in\bN \ \exists\W\in\bN$ such that $\U,\V\leq_{\F_{(\N,+)}}\W$ (for a proof of this fact, see \cite{4} or \cite{14}). We now want to prove that $\left(\beta S,\mathcal{F}_{\left(S,\cdot\right)}\right)$ is upward directed for every commutative semigroup $S$. We prove a stronger result that, when applied to a commutative semigroup $S$, provides a common $\F_{\left(S,\cdot\right)}$-upper bound to $\U,\V$ for every $\U,\V\in\beta S$. 

\begin{prop}\label{CommUpDir} Let $S$ be a semigroup and let $\U,\V\in\beta S$. Then $\U\leq_{\mathcal{F}^{R}_{\left(S,\cdot\right)}}\U\odot\V$ and $\V\leq_{\mathcal{F}^{L}_{\left(S,\cdot\right)}}\U\odot\V$. In particular, if $S$ is commutative then both $\U$ and $\V$ are $\F_{(S,\cdot)}$-finitely embeddable in $\U\odot\V$. \end{prop}

\begin{proof} Let $A\in\U\odot\V$. By definition, $\left\{s\in S\mid \left\{r\in S\mid s\cdot r\in A\right\}\in\V\right\}\in\U$. Now let 
\begin{equation*} B=\left\{s\in S\mid \left\{r\in S\mid s\cdot r\in A\right\}\in\V\right\} \end{equation*}
and, for every $s\in B$, let 
\begin{equation*}C_{s}=\left\{r\in S\mid s\cdot r\in A\right\}.\end{equation*}
It is immediate to see that $C_{s}\leq_{\mathcal{F}^{L}_{\left(S,\cdot\right)}} A$ for every $s\in B$, since $s\cdot C_{S}\subseteq A$. This shows that $\V\leq_{\mathcal{F}^{L}_{\left(S,\cdot\right)}}\U\odot\V$. To prove that $\U\leq_{\mathcal{F}^{R}_{\left(S,\cdot\right)}}\U\odot\V$ we now show that $B\leq_{\mathcal{F}^{R}_{\left(S,\cdot\right)}} A$. In fact, let $F=\left\{s_{1},...,s_{n}\right\}$ be a finite subset of $B$. Let $C=C_{s_{1}}\cap...\cap C_{s_{n}}$. $C\in\V$, so $C\neq\emptyset$ and, for every $r\in C$, by construction we have that $F\cdot r\in A$. This concludes the proof.\end{proof} 

\begin{cor}\label{genevier} If $S$ is a commutative semigroup with identity then $\left(\beta S,\mathcal{F}_{\left(S,\cdot\right)}\right)$ is upward directed. In particular $M(\beta S, \F_{(S,\cdot)})\neq\emptyset$.\end{cor}

For a generic semigroup $S$ and a generic family of functions $\F$ it is not immediate to decide if $\left(\beta S,\leq_{\mathcal{F}}\right)$ is upward directed or not. Nevertheless, in many cases that are interesting for applications the property of being upward directed can be deduced from the following result.

\begin{prop}\label{UpwDir} Let $S$ be a semigroup, let $n\geq 1$ be a natural number and let $\mathcal{F}_{1}, \mathcal{F}_{2}\subseteq \mathfrak{F}(S^{n},S)$. Let us assume that $\mathcal{F}_{1}\subseteq \mathcal{F}_{2}$. If $\left(\beta S,\leq_{\mathcal{F}_{1}}\right)$ is upward directed then $\left(\beta S,\leq_{\mathcal{F}_{2}}\right)$ is upward directed. \end{prop}

\begin{proof} We know from Proposition \ref{listona2} that if $\U\leq_{\F_{1}}\V$ then $\U\leq_{\F_{2}}\V$ for every $\U,\V\in\beta S$. Now let $\U,\V\in\bN$ and let $\W\in\beta S$ be such that $\U\leq_{F_{1}}\W$ and $\V\leq_{F_{1}}\W$. Then $\U\leq_{F_{2}}\W$ and $\V\leq_{F_{2}}\W$, so $\leq_{\F_{2}}$ is upward directed as well.\end{proof}

An immediate consequence of Proposition \ref{UpwDir} is that $\left(\bN,\leq_{\A}\right)$ is upward directed, since $\F_{\left(\N,+\right)}\subseteq\A$ and $\leq_{\F_{\left(\N,+\right)}}$ is upward directed. Moreover, since $(\A,\circ)$ is a monoid, we can apply Corollary \ref{orere} to deduce that there are maximal ultrafilters in $\left(\bN,\A\right)$ (we will study these ultrafilters in Section \ref{fruttj}).

A first immediate result on sets of maximal ultrafilters is the following.

\begin{prop}\label{papavero} Let $S$ be a semigroup, let $n\geq 1$ be a natural number and let $\F_{1}, \F_{2}\subseteq\mathfrak{F}\left(S^{n},S\right)$. If $\mathcal{F}_{1}\subseteq \mathcal{F}_{2}$ then $M\left(\beta S,\mathcal{F}_{1}\right)\subseteq M\left(\beta S,\mathcal{F}_{2}\right)$. \end{prop}

\begin{proof} For every ultrafilters $\U,\V$ in $\bN$, if $\U\leq_{\mathcal{F}_{1}}\V$ then $\U\leq_{\mathcal{F}_{2}}\V$, so we have the thesis.\end{proof}

In the next section we will show how to characterize $M\left(\beta S,\F\right)$ by means of a nonstandard characterization of $\leq_{\F}$. In particular we will be able to give a nonstandard characterization of upper cones, that are defined as follows for every semigroup $\left(S,\cdot\right)$ and every set $\mathcal{F}\subseteq \mathfrak{F}(S^{n},S)$ (even when $\mathcal{F}$ is not transitive).

\begin{defn} Let $S$ be a semigroup, let $n\geq 1$ be a natural number, let $\U\in\beta S$ and let $\F\subseteq \mathfrak{F}(S^{n},S)$. The upper cone of $\U$ in $\left(\beta S,\fe\right)$ (notation: $\mathcal{C}_{\mathcal{F}}\left(\U\right)$) is the set
\begin{equation*} \mathcal{C}_{\mathcal{F}}\left(\U\right)=\left\{\V\in\beta S\mid \U\leq_{\F}\V\right\}. \end{equation*} \end{defn}

An interesting topological property of upper cones is that they are closed in the Stone topology.

\begin{prop} Let $S$ be a semigroup and let $n\geq 1$ be a natural number. For every ultrafilter $\U\in\beta S$, for every $\F\subseteq\mathfrak{F}\left(S^{n},S\right)$ the cone $\mathcal{C}_{\mathcal{F}}\left(\U\right)$ is closed in the Stone topology. \end{prop}

\begin{proof} Let $\V\in\overline{\mathcal{C}_{\mathcal{F}}\left(\U\right)}$. Let $A\in\V$. By definition of closure in the Stone topology, there exists $\W\in\mathcal{C}_{\mathcal{F}}\left(\U\right)$ with $A\in\W$. Since $\U\fe\W$, there exists $B\in\U$ such that $B\fe A$. Since this holds for every $A\in\V$, we obtain that $\U\fe\V$, so $\V\in\mathcal{C}_{\mathcal{F}}\left(\U\right)$. This shows that $\mathcal{C}_{\mathcal{F}}\left(\U\right)$ is closed.\end{proof}

\section{Nonstandard characterizations}\label{uci}

In this section we assume that the reader knows the basics of nonstandard analysis. We refer to \cite{6}, 4.4 for the foundational aspects of nonstandard analysis and to \cite{7} for all the nonstandard notions and definitions.

\subsection{The hyperextension $^{\ast\ast}\! S$}

Following an approach similar to the one used in \cite{9}, \cite{dinasso}, \cite{14}, \cite{15}, \cite{17} we work in hyperextensions of $S$ where the star map can be iterated (the existence of such hyperextensions has been proved by V. Benci and M. Di Nasso in \cite{2}). This can be done by considering a set $X\supset S$ and a star map $\ast:\mathbb{V}\left(X\right)\rightarrow\mathbb{V}\left(X\right)$ with the transfer property, where $\mathbb{V}\left(X\right)$ is the superstructure on $X$. For our purposes it will be sufficient to consider $^{\ast\ast}\!  S$:
\begin{equation*} S\stackrel{\ast}{\longrightarrow} \!^{\ast}\! S\stackrel{\ast}{\longrightarrow} \!^{\ast\ast}\!  S. \end{equation*}
In particular, $^{\ast}\! S$ is a nonstandard extension of $S$ and $^{\ast\ast}\!  S$ is a nonstandard extension of $^{\ast}\! S$ and of $S$.\par
We will use these nonstandard extensions to identify ultrafilters and nonstandard points via the following association\footnote{To work, this association needs a technical hypothesis that we will always assume, namely that the hyperextension $^{\ast}\! S$ has the $k^{+}$-enlarging property, where $k$ is the cardinality of $\Par\left(S\right)$. See e.g. \cite{14}, \cite{15}, \cite{17}, \cite{18}, \cite{19} for details.}:

\begin{center} $(\alpha_{1},\dots,\alpha_{n})\in \!^{\ast\ast}\!  S^{n}\Rightarrow\mathfrak{U}_{(\alpha_{1},\dots,\alpha_{n})}=\left\{A\subseteq S^{n}\mid (\alpha_{1},\dots,\alpha_{n})\in \!^{\ast\ast}\! A\right\};$\\\vspace{0.3cm}
$\U\in\beta (S^{n})\Rightarrow \mu_{\U}=\left\{(\alpha_{1},\dots,\alpha_{n})\in\!^{\ast\ast}\!  S\mid \U=\mathfrak{U}_{(\alpha_{1},\dots,\alpha_{n})}\right\}$. \end{center}

It follows from the definitions that, if we set $\alpha\sim_{u}\beta\Leftrightarrow \mathfrak{U}_{\alpha}=\mathfrak{U}_{\beta}$, then $\sim_{u}$ is an equivalence relation. The set $\mu\left(\U\right)$ is called the monad of $\U$, and its elements are called generators of $\U$.\par
A simple, but important, observation is the following:
\begin{oss}\label{generatori} Let $\U\in\beta S$. Then 
	\begin{equation*} \bigcup\limits_{\W\in \mathfrak{G}\left(\U^{(n)}\right)} \mu\left(\W\right)=\mu\left(\U\right)^{n}. \end{equation*}
\end{oss}
A particularity of these iterated nonstandard extensions is that they allow to characterize many operations between ultrafilters in terms of their generators. E.g., in \cite{9}, \cite{14} it is proved that, for every $\alpha,\beta \in\!^{\ast}\! \N$,

\begin{center} $\mathfrak{U}_{\alpha}\oplus\mathfrak{U}_{\beta}=\mathfrak{U}_{\alpha+\!^{\ast}\! \beta}$,\\\vspace{0.3cm}
$\mathfrak{U}_{\alpha}\odot\mathfrak{U}_{\beta}=\mathfrak{U}_{\alpha\cdot\!^{\ast}\! \beta}$. \end{center}

It is not difficult to see that this result can be generalized to arbitrary semigroups. 

\begin{thm}\label{nonstcarul} Let $S$ be a semigroup. Let $\alpha,\beta\in\!^{\ast}\! S$. Then $\mathfrak{U}_{\alpha}\odot\mathfrak{U}_{\beta}=\mathfrak{U}_{\alpha\cdot \!^{\ast}\! \beta}$.\end{thm}

\begin{proof} By definition: $A\in\mathfrak{U}_{\alpha}\odot\mathfrak{U}_{\beta}\Leftrightarrow\left\{s\in S\mid\left\{r\in S\mid s\cdot r \in A\right\}\in\mathfrak{U}_{\beta}\right\}\in\mathfrak{U}_{\alpha}\Leftrightarrow\alpha\in\!^{\ast}\! \left\{s\in S\mid \beta\in\!^{\ast}\! \left\{r\in S\mid s\cdot r\in A\right\}\right\}\Leftrightarrow \alpha\cdot\!^{\ast}\! \beta\in\!^{\ast\ast}\! A$. \end{proof}

This nonstandard characterization of ultrafilters and operations (for $(S,\cdot)=(\N,+), (\N,\cdot)$) has been used in \cite{9}, \cite{14}, \cite{15} and \cite{17} to study problems in combinatorial number theory on $\N$ related to the partition regularity of certain (linear and nonlinear) equations. See also \cite{dinasso}, where the use of nonstandard analysis in various aspects of combinatorics is presented by means of some examples. In the next section we will show how it can be used to characterize the cones $\mathcal{C}_{\F}\left(\U\right)$.

\subsection{Nonstandard Characterization of $\mathcal{C}_{\F}\left(\U\right)$}

In this section we do not make any assumption on $\F$, so we do not assume $\F$ to be transitive or $\fe$ to be upward directed.\par
In \cite{4} the following nonstandard characterization of $\leq_{\F_{\left(\N,+\right)}}$ has been proved by A. Blass and M. Di Nasso (see also \cite{8} for a similar result regarding the finite embeddability between sets of integers).

\begin{prop}[\cite{4}, Proposition 16]\label{nsfe} Let $A,B\subseteq \N$. Then $A\leq_{\F_{\left(\N,+\right)}} B$ if and only if there exists $\alpha\in\!^{\ast}\! \N$ such that $\alpha+A\subseteq\!^{\ast}\! B$.\end{prop}

It is possible to generalize Proposition \ref{nsfe} to every semigroup $S$ and every family of functions $\mathcal{F}\subseteq \mathfrak{F}\left(S^{n},S\right)$.

\begin{prop}\label{nonstandchar} Let $S$ be a semigroup, let $n\geq 1$ be a natural number, let $A,B$ be subsets of $S$ and let $\mathcal{F}\subseteq \mathfrak{F}\left(S^{n},S\right)$. The following two conditions are equivalent:
\begin{enumerate}[leftmargin=*,label=(\roman*),align=left ]
	\item $A\leq_{\mathcal{F}}B$;
	\item there is a function $\varphi$ in $^{\ast}\! \mathcal{F}$ such that $\varphi\left(A^{n}\right)\subseteq\!^{\ast}\! B$.
\end{enumerate}
\end{prop}

\begin{proof} For every finite subset $F$ of $A$, let us consider the set
\begin{equation*} R_{F}=\left\{f\in \mathcal{F}\mid f\left(F^{n}\right)\subseteq B\right\}. \end{equation*}
Since $A\fe B$ the family $\left\{R_{F}\right\}_{F\in\Par_{fin}\left(A\right)}$ has the finite intersection property\footnote{Since, as we already pointed out, we assume that the hyperextensions that we are working with satisfy the $|\Par(S)|^{+}$-enlarging property.}, so
\begin{equation*} \bigcap\limits_{F\in\Par_{fin}\left(A\right)}\!^{\ast}\! R_{F} \neq\emptyset. \end{equation*}
Let $\varphi$ be a function in this intersection. By construction and transfer, $\varphi$ has the following two properties:
\begin{enumerate}
	\item $\varphi\in\!^{\ast}\! \mathcal{F}$;
	\item $\varphi\left(\:^{\ast}\! F^{n}\right)\subseteq\!^{\ast}\! B$ for every finite subset $F$ of $A$.
\end{enumerate}
As $^{\ast}\! F^{n}=F^{n}$ for every finite subset of $ S$, by condition (2) it follows that $\varphi\left(A^{n}\right)\subseteq\!^{\ast}\! B$.\par
Conversely, let $\varphi$ be a function in $^{\ast}\! \mathcal{F}$ such that $\varphi\left(A^{n}\right)\subseteq\!^{\ast}\! B$. By contrast, let us suppose that $A$ is not $\mathcal{F}$-finitely embeddable in $B$. Let $F$ be a finite subset of $A$ such that, for every function $g\in\mathcal{F}$, $g\left(F^{n}\right)$ is not included in $B$. By transfer it follows that, for every function $\psi\in\!^{\ast}\! \mathcal{F}$, $\psi\left(\:^{\ast}\! F^{n}\right)$ is not included in $^{\ast}\! B$, and this is absurd because $^{\ast}\! F^{n}=F^{n}\subseteq A$ and $\varphi\left(A^{n}\right)\subseteq \!^{\ast}\! B$.\end{proof}

Let us observe that Proposition \ref{nsfe} is a consequence of Proposition \ref{nonstandchar}, since
\begin{equation*} ^{\ast}\! \F_{\left(\N,+\right)}=\left\{t_{\alpha}:\!^{\ast}\!  \N\rightarrow\!^{\ast}\!  \N\mid t_{\alpha}\left(\eta\right)=\alpha+\eta \ \forall\eta\in\!^{\ast}\!  \N\right\}, \end{equation*}
so $t_{\alpha}\left(A\right)=\alpha+A$ for every $A\subseteq\N$.\par
For our purposes it is important to consider the hyperextensions $^{\ast}\! \varphi$ of internal functions $\varphi:\!^{\ast}\! S^{n}\rightarrow\!^{\ast}\! S$; in particular, we will use the following property. 

\begin{prop}\label{pippo} Let $S$ be a semigroup, let $n\geq 1$ be a natural number, let $\varphi\in\!^{\ast}\! \left(\mathfrak{F}\left(S^{n},S\right)\right)$ be an internal function and let $\alpha_{1},\dots,\alpha_{n},$ $\beta_{1},\dots,\beta_{n}\in\!^{\ast}\!  S^{n}$. If $\left(\alpha_{1},\dots,\alpha_{n}\right)\sim_{u}\left(\beta_{1},\dots,\beta_{n}\right)$ then $\left(\:^{\ast}\! \varphi\right)\left(\alpha_{1},\dots\alpha_{n}\right)\sim_{u}\left(\:^{\ast}\! \varphi\right)\left(\beta_{1},\dots,\beta_{n}\right)$. \end{prop}

\begin{proof} Let $A\subseteq S$. Since $\left(\alpha_{1},\dots\alpha_{n}\right)\sim_{u}\left(\beta_{1},\dots,\beta_{n}\right)$, we have that
\begin{center} $\left(\alpha_{1},\dots\alpha_{n}\right)\in\!^{\ast}\! \left\{\left(s_{1},\dots,s_{n}\right)\in S^{n}\mid \varphi\left(s_{1},\dots,s_{n}\right)\in\!^{\ast}\! A\right\}\Leftrightarrow $\\\vspace{0.3cm}
 $\left(\beta_{1},\dots,\beta_{n}\right)\in\!^{\ast}\! \left\{\left(s_{1},\dots,s_{n}\right)\in S^{n}\mid \varphi\left(s_{1},\dots,s_{n}\right)\in\!^{\ast}\! A\right\}.$ \end{center}
Then
\begin{center} $\left(^{\ast}\! \varphi\right)\left(\alpha_{1},\dots\alpha_{n}\right)\in\!^{\ast\ast}\! A\Leftrightarrow$\\\vspace{0.3cm}
$(\alpha_{1},\dots,\alpha_{n})\in\!^{\ast}\! \left\{\left(s_{1},\dots,s_{n}\right)\in S^{n}\mid \varphi\left(s_{1},\dots,s_{n}\right)\in\!^{\ast}\! A\right\}\Leftrightarrow$\\\vspace{0.3cm}
$\left(\beta_{1},\dots,\beta_{n}\right)\in\!^{\ast}\! \left\{\left(s_{1},\dots,s_{n}\right)\in S^{n}\mid \varphi\left(s_{1},\dots,s_{n}\right)\in\!^{\ast}\! A\right\}\Leftrightarrow$\\\vspace{0.3cm}
$\left(\:^{\ast}\! \varphi\right)\left(\beta_{1},\dots,\beta_{n}\right)\in\!^{\ast\ast}\! A,$ \end{center}

so $\left(\:^{\ast}\! \varphi\right)\left(\alpha_{1},\dots\alpha_{n}\right)\sim_{u}\left(\:^{\ast}\! \varphi\right)\left(\beta_{1},\dots,\beta_{n}\right)$.\end{proof}

Equivalently, Proposition \ref{pippo} can be restated by saying that for every internal function $\varphi\in\!^{\ast}\! \left(\mathfrak{F}\left(S^{n},S\right)\right)$ and for every elements $\left(\alpha_{1},\dots,\alpha_{n}\right),\left(\beta_{1},\dots,\beta_{n}\right)\in\!^{\ast}\! S^{n}$ we have that

\begin{center} if $\mathfrak{U}_{\left(\alpha_{1},\dots,\alpha_{n}\right)}=\mathfrak{U}_{\left(\beta_{1},\dots,\beta_{n}\right)}$ then $\mathfrak{U}_{\left(^{\ast}\! \varphi\right)\left(\alpha_{1},\dots,\alpha_{n}\right)}=\mathfrak{U}_{\left(^{\ast}\! \varphi\right)\left(\beta_{1},\dots,\beta_{n}\right)}$. \end{center}

Now, given any internal function $\varphi\in\!^{\ast}\! \left(\mathfrak{F}\left(S^{n},S\right)\right)$, let $\overline{\varphi}:\beta \left(S^{n}\right)\rightarrow\beta S$ be the function such that
\begin{equation*} \overline{\varphi}\left(\mathfrak{U}_{\left(\alpha_{1},\dots,\alpha_{n}\right)}\right)=\mathfrak{U}_{^{\ast}\! \varphi\left(\alpha_{1},\dots,\alpha_{n}\right)} \end{equation*}
for every $\left(\alpha_{1},\dots,\alpha_{n}\right)\in\!^{\ast}\! S^{n}$. This is the extension to internal functions of the association $f\in \mathfrak{F}\left(S^{n},S\right)\rightarrow\overline{f}\in\mathfrak{F}\left(\beta\left(S^{n}\right),\beta S\right)$. In fact we have the following:

\begin{prop} Let $S$ be a semigroup, let $n\geq 1$ be a natural number and let $\varphi\in\!^{\ast}\!\left(\mathfrak{F}\left(S^{n},S\right)\right)$. If $\varphi=\!^{\ast}\! g$ for some standard function $g\in \mathfrak{F}\left(S^{n},S\right)$ then
\begin{equation*}\overline{\varphi}=\overline{g}. \end{equation*}
\end{prop}

\begin{proof} Let us notice that, for every set $A\subseteq S$ and for every $\left(\alpha_{1},\dots,\alpha_{n}\right)\in\!^{\ast}\!  S^{n}$, we have that  
\begin{center} $A\in\overline{g}\left(\mathfrak{U}_{\left(\alpha_{1},\dots,\alpha_{n}\right)}\right)\Leftrightarrow g^{-1}(A)\in \mathfrak{U}_{\left(\alpha_{1},\dots,\alpha_{n}\right)}\Leftrightarrow$\\\vspace{0.3cm}
$\left(\alpha_{1},\dots,\alpha_{n}\right)\in\!^{\ast}\!\left(g^{-1}(A)\right)\Leftrightarrow\!^{\ast}\!g\left(\alpha_{1},\dots,\alpha_{n}\right)\in\!^{\ast}\!A\Leftrightarrow$\\\vspace{0.3cm}
$A\in\mathfrak{U}_{\!^{\ast}\! g\left(\alpha_{1},\dots,\alpha_{n}\right)}$.\end{center}
Moreover, the function $\varphi=\!^{\ast}\!g$ satisfies the following property:
\begin{equation*} \forall \left(s_{1},\dots,s_{n}\right)\in S^{n} \ \varphi\left(s_{1},\dots,s_{n}\right)=\left(\:^{\ast}\! g\right)\left(s_{1},\dots,s_{n}\right)=g\left(s_{1},\dots,s_{n}\right), \end{equation*}

hence, by transfer, we have that
\begin{equation*} \forall \left(\eta_{1},\dots,\eta_{n}\right)\in\!^{\ast}\!  S^{n},\ \left(\:^{\ast}\! \varphi\right)\left(\eta_{1},\dots,\eta_{n}\right)=\left(\:^{\ast\ast}\! g\right)\left(\eta_{1},\dots,\eta_{n}\right)=\left(\:^{\ast}\! g\right)\left(\eta_{1},\dots,\eta_{n}\right), \end{equation*}
so $\overline{g}\left(\mathfrak{U}_{\left(\alpha_{1},\dots,\alpha_{n}\right)}\right)=\mathfrak{U}_{^{\ast}\! g\left(\alpha_{1},\dots,\alpha_{n}\right)}=\mathfrak{U}_{\varphi\left(\alpha_{1},\dots,\alpha_{n}\right)}=\mathfrak{U}_{^{\ast}\! \varphi\left(\alpha_{1},\dots,\alpha_{n}\right)}=\overline{\varphi}\left(\mathfrak{U}_{\left(\alpha_{1},\dots,\alpha_{n}\right)}\right)$.\end{proof}

\begin{cor}\label{pippa} Let $S$ be a semigroup, let $n\geq 1$ be a natural number, let $A,B\subseteq S$ and let $\F\subseteq\mathfrak{F}\left(S^{n},S\right)$. The following two conditions are equivalent:
\begin{enumerate}[leftmargin=*,label=(\roman*),align=left ]
	\item $A\leq_{\mathcal{F}}B$;
	\item there is a function $\varphi\in\!^{\ast}\! \mathcal{F}$ such that, for every ultrafilter $\U$ in $\beta \left(S^{n}\right)$ with $A^{n}\in\U$, for every $\left(\alpha_{1},\dots,\alpha_{n}\right)\in\!^{\ast}\! S^{n}$ with $\U=\mathfrak{U}_{\left(\alpha_{1},\dots,\alpha_{n}\right)}$, we have that $B\in\mathfrak{U}_{\left( ^{\ast}\! \varphi\right)\left(\alpha_{1},\dots,\alpha_{n}\right)}$.
\end{enumerate}
\end{cor}

\begin{proof} From Proposition \ref{nonstandchar} we know that $A\leq_{\mathcal{F}} B$ if and only if there is a function $\varphi\in\!^{\ast}\! \mathcal{F}$ such that $\varphi\left(A^{n}\right)\subseteq\!^{\ast}\! B$. By transfer, $\varphi\left(A^{n}\right)\subseteq\!^{\ast}\! B$ if and only if $^{\ast}\! \varphi\left(\:^{\ast}\! A^{n}\right)\subseteq\!^{\ast\ast}\! B$. So:
\begin{center} $A\leq_{\mathcal{F}}B\Leftrightarrow\left(\exists\varphi\in\!^{\ast}\! \mathcal{F}\right)\left(\:^{\ast}\! \varphi\left(\:^{\ast}\! A^{n}\right)\subseteq\!^{\ast\ast}\! B\right)\Leftrightarrow$\\\vspace{0.3cm} 
$\left(\exists\varphi\in\!^{\ast}\! \mathcal{F}\right)\left(\forall \left(\alpha_{1},\dots,\alpha_{n}\right)\in\!^{\ast}\!  S^{n}\right)\left(\left(\alpha_{1},\dots,\alpha_{n}\right)\in\!^{\ast}\! A^{n}\Rightarrow\!^{\ast}\! \varphi\left(\alpha_{1},\dots,\alpha_{n}\right)\in\!^{\ast\ast}\! B\right)\Leftrightarrow$\\\vspace{0.3cm}
$\left(\exists \varphi\in\!^{\ast}\! \mathcal{F}\right)\left(\forall \left(\alpha_{1},\dots,\alpha_{n}\right)\in\!^{\ast}\!  S^{n}\right)\!\left(A^{n}\in\mathfrak{U}_{\left(\alpha_{1},\dots,\alpha_{n}\right)}\Rightarrow B\in\mathfrak{U}_{^{\ast}\! \varphi\left(\alpha_{1},\dots,\alpha_{n}\right)}\right).\qedhere$ \end{center}
\end{proof}
To deduce a nonstandard characterization of the cones from the previous nonstandard characterization of $\mathcal{F}$-finite embeddabilities we need one last result.

\begin{lem}\label{miao} Let $S$ be a semigroup, let $n\geq 1$ be a natural number, let $\U\in\beta S$, let $\F\subseteq\mathfrak{F}\left(S^{n},S\right)$ and let $B\in\Par\left( S\right)$. The following two conditions are equivalent:
\begin{enumerate}[leftmargin=*,label=(\roman*),align=left ]
	\item\label{frac} there is a set $A$ in $\U$ such that $A\leq_{\mathcal{F}} B$;
	\item\label{trac} there is a function $\varphi$ in $^{\ast}\! \mathcal{F}$ such that $B\in \overline{\varphi}\left(\W\right)$ for every $\W\in \mathfrak{G}\left(\U^{\left(n\right)}\right)$.
\end{enumerate}
\end{lem}

\begin{proof} The implication $\ref{frac}\Rightarrow \ref{trac}$ is a consequence of Corollary \ref{pippa}, since by construction $A^{n}\in\W$ for every $\W\in \mathfrak{G}\left(\U^{\left(n\right)}\right)$.\par 
Conversely, let us suppose that $\forall A\in\U$ $A\nleq_{\F} B$. From Proposition \ref{nonstandchar} this is equivalent to say that
\begin{equation*} \forall A\in\U \ \forall\varphi\in\!^{\ast}\!\F \ \exists a_{1},\dots,a_{n}\in A \ \text{s.t.} \ \varphi\left(a_{1},\dots,a_{n}\right)\notin\!^{\ast}\!B. \end{equation*}
Let $\varphi$ be given as in \ref{trac}. For every $A\in\U$ let
\begin{equation*} A_{\varphi}=\left\{\left(a_{1},\dots,a_{n}\right)\in A^{n}\mid \varphi\left(a_{1},\dots,a_{n}\right)\notin\!^{\ast}\!B\right\}. \end{equation*}
The family $\left\{A_{\varphi}\right\}_{A\in\U}$ has the finite intersection property, so there is $\left(\alpha_{1},\dots,\alpha_{n}\right)\in\bigcap\limits_{A\in\U}\!^{\ast}\!A_{\varphi}$. By construction
\begin{itemize}
	\item $\left(\alpha_{1},\dots,\alpha_{n}\right)\in \mu_{\U}^{n}$;
	\item $^{*}\!\varphi\left(\alpha_{1},\dots,\alpha_{n}\right)\notin\!^{\ast\ast}\!B$.
\end{itemize}
From Observation \ref{generatori} we deduce that, if we set $\W=\mathfrak{U}_{\left(\alpha_{1},\dots,\alpha_{n}\right)}$, then $\W\in \mathfrak{G}\left(\U^{\left(n\right)}\right)$ and $B\notin \overline{\varphi}\left(\W\right)$, which is absurd.\end{proof}

We can now state the desired nonstandard characterization of the upper cones $\mathcal{C}_{\mathcal{F}}\left(\U\right)$.

\begin{thm}\label{questoqui} Let $S$ be a semigroup. For every $\U\in\beta S$, for every natural number $n\geq 1$ and for every $\F\subseteq\mathfrak{F}\left(S^{n},S\right)$ we have
\begin{center} $\mathcal{C}_{\mathcal{F}}\left(\U\right)=\overline{\left\{\mathfrak{U}_{\left( ^{\ast}\! \varphi\right)\left(\alpha_{1},\dots,\alpha_{n}\right)}\mid \varphi\in\:^{\ast}\! \mathcal{F}, \left(\alpha_{1},\dots,\alpha_{n}\right)\in \mu\left(\U\right)^{n}\right\}}=$\\\vspace{0.3cm}
$\overline{\left\{\overline{\varphi}\left(\W\right)\mid \varphi\in\! ^{\ast}\! \mathcal{F}, \W\in \mathfrak{G}\left(\U^{\left(n\right)}\right)\right\}},$ \end{center} 
where the closure is taken in the Stone topology.\end{thm}

\begin{proof} Let $\V\in\mathcal{C}_{\mathcal{F}}\left(\U\right)$; by definition, for every set $B$ in $\V$ there is a set $A$ in $\U$ such that $A\leq_{\mathcal{F}} B$. From Lemma \ref{miao} we deduce that there is a function $\varphi$ in $^{\ast}\! \mathcal{F}$ such that $B\in \overline{\varphi}\left(\W\right)$ for every $\W\in \mathfrak{G}\left(\U^{\left(n\right)}\right)$. \par
So $\V\in\overline{\left\{\overline{\varphi}\left(\W\right)\mid \varphi\in \!^{\ast}\! \mathcal{F}, \W\in \mathfrak{G}\left(\U^{\left(n\right)}\right)\right\}}$, hence
\begin{equation*} \mathcal{C}_{\mathcal{F}}\left(\U\right)\subseteq\overline{\left\{\overline{\varphi}\left(\W\right)\mid \varphi\in\! ^{\ast}\! \mathcal{F}, \W\in \mathfrak{G}\left(\U^{\left(n\right)}\right)\right\}}.\end{equation*}
Conversely, let $\V\in\overline{\left\{\overline{\varphi}\left(\W\right)\mid \varphi\in\! ^{\ast}\! \mathcal{F}, \W\in \mathfrak{G}\left(\U^{\left(n\right)}\right)\right\}}$. Then for every set $B$ in $\V$ there is a function $\varphi$ in $^{\ast}\! \mathcal{F}$ and a ultrafilter $\W\in \mathfrak{G}\left(\U^{\left(n\right)}\right)$ such that $B\in\overline{\varphi}\left(\W\right)$; from Lemma \ref{miao} it follows that there is a set $A$ in $\U$ such that $A\leq_{\mathcal{F}} B$. So $\U\fe \V$ and hence
\begin{equation*} \overline{\left\{\overline{\varphi}\left(\W\right)\mid \varphi\in\! ^{\ast}\! \mathcal{F}, \W\in \mathfrak{G}\left(\U^{\left(n\right)}\right)\right\}}\subseteq \mathcal{C}_{\mathcal{F}}\left(\U\right).\qedhere \end{equation*}
\end{proof}

\begin{cor}\label{rughe} For every semigroup $S$, for every ultrafilter $\U\in\beta S$ we have that
\begin{equation*}\mathcal{C}_{\F^{R}_{\left(S,\cdot\right)}}\left(\U\right)=\overline{\left\{\U\odot\V\mid\V\in\beta S\right\}}. \end{equation*}
\end{cor}
\begin{proof} Let $\U\in\beta S$. Notice that, by definition, $\mathfrak{G}\left(\U^{(1)}\right)=\U$. Since 
\begin{equation*} \F^{R}_{\left(S,\cdot\right)}=\left\{f_{r}\in \mathfrak{F}\left(S,S\right)\mid r\in S \ \text{and} \ f_{r}\left(s\right)=s\cdot r \ \forall s\in \N\right\}, \end{equation*}
we have that
\begin{equation*} ^{\ast}\! \F^{R}_{\left(S,\cdot\right)}=\left\{f_{\eta}\in\!^{\ast}\!\left(\mathfrak{F}\left(S,S\right)\right)\mid\eta\in\!^{\ast}\! S \ \text{and} \ f_{\eta}\left(\mu\right)=\mu\cdot\eta \ \forall\mu\in\!^{\ast}\! S\right\}. \end{equation*}
If $f_{\eta}\in\!^{\ast}\! \F^{R}_{\left(S,\cdot\right)}$, by transfer $^{\ast}\! f_{\eta}:\!^{\ast\ast}\! S\rightarrow\!^{\ast\ast}\! S$ is the function such that, for every $\alpha\in\!^{\ast\ast}\! S$,
\begin{equation*} \left(\:^{\ast}\! f_{\eta}\right)\left(\alpha\right)=\alpha\cdot\!^{\ast}\! \eta. \end{equation*}
Hence we deduce from Theorem \ref{questoqui} that, if $\U=\mathfrak{U}_{\alpha}$, then
\begin{equation*} \mathcal{C}_{\F^{R}_{\left(S,\cdot\right)}}\left(\U\right)=\overline{\left\{\mathfrak{U}_{\alpha\cdot^{\ast}\! \mu}\mid \mu\in\!^{\ast}\! S\right\}}. \end{equation*}
Since $\mathfrak{U}_{\alpha\cdot\!^{\ast}\! \mu}=\mathfrak{U}_{\alpha}\odot\mathfrak{U}_{\mu}$ we conclude that
\begin{equation*} \mathcal{C}_{\F^{R}_{\left(S,\cdot\right)}}\left(\U\right)=\overline{\left\{\mathfrak{U}_{\alpha}\odot\mathfrak{U}_{\mu}\mid \mathfrak{U}_{\mu}\in\beta S\right\}}=\overline{\left\{\U\odot\V\mid\V\in\beta S\right\}}.\qedhere \end{equation*} \end{proof}

\subsection{Generating functions}\label{trevisan}

The characterization of cones given by Theorem \ref{questoqui} can be simplified for many choices of the set of functions $\mathcal{F}$, in analogy to Corollary \ref{rughe}. The central idea for this semplication is that of "generating function of $\mathcal{F}$", that we now define.

\begin{defn} Let $S$ be a semigroup, let $G\in \mathfrak{F}\left(S^{n}\times S^{k},S\right)$ and let $R$ be a subset of $S^{k}$. The set of functions generated by the pair $\left(G,R\right)$ is the set
\begin{multline*} \mathcal{F}\left(G,R\right)=\Big\{f_{r_{1},\dots,r_{k}}\left(s_{1},\dots,s_{n}\right)\in \mathfrak{F}\left(S^{n},S\right)\mid  \left(r_{1},\dots,r_{k}\right)\in R\ \text{and} \\
\forall \left(s_{1},\dots,s_{n}\right)\in S^{n} \ f_{r_{1},\dots,r_{k}}\left(s_{1},\dots,s_{n}\right)=G\left(\left(s_{1},\dots,s_{n}\right),\left(r_{1},\dots,r_{k}\right)\right)\Big\}. \end{multline*}
$G$ is called the generating function of $\mathcal{F}\left(G,R\right)$, and $R$ is called the set of parameters of $\mathcal{F}\left(G,R\right)$. Whenever $\mathcal{F}=\mathcal{F}\left(G,R\right)$ for some $G,R$ we will also say that $\mathcal{F}$ is generated by the pair $\left(G,R\right)$.

\end{defn}

Most interesting families of functions $\mathcal{F}$ are generated by appropriate pairs $\left(G,R\right)$. In particular, for every semigroup $S$, if $G_{L}, G_{R}:S\times S\rightarrow S$ satisfy 
\begin{equation*} G_{L}\left(s,r\right)=s\cdot r,\  G_{R}\left(s,r\right)=r\cdot s \end{equation*}
then it is immediate to see that 
\begin{equation*} \mathcal{F}^{L}_{S}=\mathcal{F}\left(G_{L},S\right) \ \text{and} \ \mathcal{F}^{R}_{S}=\mathcal{F}\left(G_{R},S\right). \end{equation*}
However there are many more families generated by a pair. We list three examples for $S=\N$ in the following table.\\

\begin{longtable}{l|l} 
\

{\bfseries Set of Functions} & {\bfseries Generating Function, Set of Parameters}\\
			 & \\
			$\mathcal{A}$ & $G\left(n,\left(a,b\right)\right)=an+b$, $R=\N^{2}\setminus \left\{\left(0,b\right)\mid b\in\N\right\}$\\ 
			 & \\
			$\mathcal{G}$ & $G\left(\left(n,m\right),\left(r,a,b\right)\right)=r^{n}\left(a+mb\right)$,\\
			& $R=\left\{\left(r,a,b\right)\in\N^{3}\mid r>1, b>0\right\}$\\
			 & \\
			Non-constant polynomials & $G\left(n,\left(a_{0},\dots,a_{m}\right)\right)=\sum_{i=0}^{m}a_{i}n^{i}$, \\
			with degree $m$ & $R=\N^{m+1}\setminus\left\{\left(a_{0},0,0,\dots,0\right)\mid a_{0}\in\N\right\}$

\end{longtable}
When the set $\mathcal{F}$ is generated by a pair we can rephrase Theorem \ref{questoqui} to obtain an algebraical characterization of cones. We will use the following notation: given $R\subseteq S^{k}$, we will denote by $\Theta_{R}$ the set
\begin{equation*} \Theta_{R}=\left\{\U\in\beta\left(S^{k}\right)\mid R\in\U\right\}. \end{equation*}
We also recall that, given ultrafilters $\U\in\beta\left(S^{n}\right), \V\in\beta\left(S^{k}\right)$, $\U\otimes\V$ is the ultrafilter on $S^{n+k}$ defined by the following condition: for every $A\subseteq S^{n+k}$, $A\in\U\otimes\V$ if and only if
\begin{equation*} \left\{\left(s_{1},\dots,s_{n}\right)\in S^{n}\mid\left\{\left(r_{1},\dots,r_{k}\right)\in S^{k}\mid \left(s_{1},\dots,s_{h},r_{1},\dots,r_{k}\right)\in A\right\}\in\V\right\}\in\U. \end{equation*}

\begin{thm}\label{gene} Let $G: S^{n}\times S^{k}\rightarrow S$, let $R$ be a nonempty subset of $S^{k}$, let $\U$ be an ultrafilter on $S$ and let us consider $\mathcal{F}=\mathcal{F}\left(G,R\right)$. Then
\begin{equation*} \mathcal{C}_{\mathcal{F}}\left(\U\right)=\overline{\left\{\overline{G}\left(\W\otimes\V\right)\mid \W\in \mathfrak{G}\left(\U^{\left(n\right)}\right), \V\in\Theta_{R}\right\}}. \end{equation*}
\end{thm}

\begin{proof} From Theorem \ref{questoqui} we know that  

\begin{equation}\label{guappo} \mathcal{C}_{\mathcal{F}}\left(\U\right)=\overline{\left\{\mathfrak{U}_{\left( ^{\ast}\! \varphi\right)\left(\alpha_{1},\dots,\alpha_{n}\right)}\mid \varphi\in\:^{\ast}\! \mathcal{F}, \left(\alpha_{1},\dots,\alpha_{n}\right)\in \mu\left(\U\right)^{n}\right\}}. \end{equation}

Let us notice that
\begin{multline*} ^{\ast}\! \mathcal{F}=\!^{\ast}\! \mathcal{F}\left(G,R\right)=\!^{\ast}\! \Big\{f_{r_{1},\dots,r_{k}}\in \mathfrak{F}(S^{n},S)\mid \ r_{1},\dots,r_{k}\in R \ \text{and}\\
\forall\left(s_{1},\dots,s_{n}\right)\in S^{n}\ f_{r_{1},\dots,r_{k}}\left(s_{1},\dots,s_{n}\right)=G\left(\left(s_{1},\dots,s_{n}\right),\left(r_{1},\dots,r_{k}\right)\right)\Big\}=\\
\Big\{\varphi_{\beta_{1},\dots,\beta_{k}}\in\!^{\ast}\! \left(\mathfrak{F}\left(S^{n},S\right)\right)\mid \ \beta_{1},\dots,\beta_{k}\in\!^{\ast}\! R \ \text{and} \\
\forall\left(\alpha_{1},\dots,\alpha_{n}\right)\in\!^{\ast}\! S^{n}\ \varphi_{\beta_{1},\dots,\beta_{k}}\left(\alpha_{1},\dots,\alpha_{n}\right)=\!^{\ast}\! G\left(\left(\alpha_{1},\dots,\alpha_{n}\right),\left(\beta_{1},\dots,\beta_{k}\right)\right)\Big\}. 
\end{multline*}
We observe that, by definition of set of generators, for every ultrafilter $\V$ in $\beta \left(S^{k}\right)$ we have that $\V\in\Theta_{R}$ if and only if there is a $k$-tuple $\left(\beta_{1},\dots,\beta_{k}\right)$ in $^{\ast}\! R^{k}$ such that $\V=\mathfrak{U}_{\left(\beta_{1},\dots,\beta_{k}\right)}$.\par\vspace{0.3cm}
{\bfseries Claim:} For every $k$-tuple $\left(\beta_{1},\dots,\beta_{k}\right)$ in $^{\ast}\! R^{k}$ and for every $\left(\alpha_{1},\dots,\alpha_{n}\right)\in\! ^{\ast}\! S^{n}$ we have
\begin{equation*} \mathfrak{U}_{\left(^{\ast}\! \varphi_{\beta_{1},\dots,\beta_{k}}\right)\left(\alpha_{1},\dots,\alpha_{n}\right)}=\overline{G}\left(\mathfrak{U}_{\left(\alpha_{1},\dots,\alpha_{n}\right)}\otimes\mathfrak{U}_{\left(\beta_{1},\dots,\beta_{k}\right)}\right). \end{equation*}
Let us suppose that the claim has been proved. Then 
\begin{multline*}\left\{\mathfrak{U}_{\left(^{\ast}\! \varphi\right)\left(\alpha_{1},\dots,\alpha_{n}\right)}\mid \varphi\in\!^{\ast}\! \mathcal{F}, \left(\alpha_{1},\dots,\alpha_{n}\right)\in\mu\left(\U\right)^{n}\right\}=\\\left\{\overline{G}\left(\W\otimes\V\right)\mid \V\in\Theta_{R}, \W\in \mathfrak{G}\left(\U^{\left(n\right)}\right)\right\},\end{multline*}
and the thesis follows by equation (\ref{guappo}).\par
To prove the claim, let $A$ be a subset of $S$ and let $\left(\alpha_{1},\dots,\alpha_{n}\right)\in\!^{\ast}\! S^{n}$. To simplify the notations, let $\stackrel{\rightarrow}{\alpha}=\left(\alpha_{1},\dots,\alpha_{n}\right)$, $\stackrel{\rightarrow}{s}=\left(s_{1},\dots,s_{n}\right)$, $\stackrel{\rightarrow}{\beta}=\left(\beta_{1},\dots,\beta_{k}\right)$ and $\stackrel{\rightarrow}{b}=\left(b_{1},\dots,b_{k}\right)$. Then:
\begin{center} $A\in\overline{G}\left(\mathfrak{U}_{\stackrel{\rightarrow}{\alpha}}\otimes\mathfrak{U}_{\stackrel{\rightarrow}{\beta}}\right)\Leftrightarrow $\\\vspace{0.3cm}
$\left\{\stackrel{\rightarrow}{s}\in S^{n}\mid \left\{\stackrel{\rightarrow}{b}\in S^{k}\mid G\left(\stackrel{\rightarrow}{s},\stackrel{\rightarrow}{b}\right)\in A\right\}\in\mathfrak{U}_{\stackrel{\rightarrow}{\beta}}\right\}\in\mathfrak{U}_{\stackrel{\rightarrow}{\alpha}}\Leftrightarrow$\\\vspace{0.3cm}
$\left\{\stackrel{\rightarrow}{s}\in S^{n}\mid \stackrel{\rightarrow}{\beta}\in\!^{\ast}\! \left\{\stackrel{\rightarrow}{b}\in S^{k}\mid G\left(\stackrel{\rightarrow}{s},\stackrel{\rightarrow}{b}\right)\in A\right\}\right\}\in\mathfrak{U}_{\stackrel{\rightarrow}{\alpha}}\Leftrightarrow$\\\vspace{0.3cm}
$\left\{\stackrel{\rightarrow}{s}\in S^{n}\mid \!^{\ast}\! G\left(\stackrel{\rightarrow}{s},\stackrel{\rightarrow}{\beta}\right)\in\!^{\ast}\! A\right\}\in\mathfrak{U}_{\stackrel{\rightarrow}{\alpha}}\Leftrightarrow$\\\vspace{0.3cm}
$\stackrel{\rightarrow}{\alpha}\in\!^{\ast}\! \left\{\stackrel{\rightarrow}{s}\in S^{n}\mid\!^{\ast}\! G\left(\stackrel{\rightarrow}{s},\stackrel{\rightarrow}{\beta}\right)\in\!^{\ast}\! A\right\}\Leftrightarrow\!^{\ast\ast}\! G\left(\stackrel{\rightarrow}{\alpha},\:^{\ast}\!\stackrel{\rightarrow}{\beta}\right)\in\!^{\ast\ast}\! A\Leftrightarrow$\\\vspace{0.3cm}
$\left(\:^{\ast}\! \varphi_{\stackrel{\rightarrow}{\beta}}\right)\left(\stackrel{\rightarrow}{\alpha}\right)\in\!^{\ast\ast}\! A\Leftrightarrow A\in\mathfrak{U}_{\left(^{\ast}\! \varphi_{\stackrel{\rightarrow}{\beta}}\right)\left(\stackrel{\rightarrow}{\alpha}\right)},$ \end{center} 
so the claim is proved.\end{proof}

Let us note that the result of Theorem \ref{gene} simplifies when $n=1$. In fact, in this case $\mathfrak{G}\left(\U^{\left(1\right)}\right)=\U$, therefore
\begin{equation*} \mathcal{C}_{\mathcal{F}}\left(\U\right)=\overline{\left\{\overline{G}\left(\U\otimes\V\right)\mid \V\in\Theta_{R}\right\}}. \end{equation*}

{\bfseries Example:} Let $\F$ be the following family of functions:
\begin{equation*} \F=\left\{h_{p}:\N\rightarrow\N\mid p \ \text{is prime and} \ h_{p}\left(m\right)=m^{p} \ \forall m\in\N\right\}.\end{equation*} 
The generating function of $\F$ is $G\left(m,p\right)=m^{p}$, and its set of parameters is $P=\left\{p\in\N\mid p \ \text{is prime}\right\}$. So
\begin{equation*} \mathcal{C}_{\F}\left(\U\right)=\overline{\left\{\U^{\V}\mid\V\in\Theta_{P}\right\}}=\overline{\left\{\U^{\V}\mid\V\in\bN \ \text{and} \ P\in\V\right\}}, \end{equation*}
where 
\begin{equation*} \U^{\V}=\left\{A\subseteq \N\mid \left\{n\in\N\mid\left\{m\in\N\mid n^{m}\in A\right\}\in\V\right\}\in\U\right\}. \end{equation*}

From Theorem \ref{gene} we can deduce a characterization of $M\left(\beta S,\mathcal{F}_{\left(S,\cdot\right)}\right)$ for every commutative semigroup $S$, which generalizes the main result in \cite{16}.

\begin{thm}\label{braiati} Let $S$ be a commutative semigroup with identity. Then

\begin{equation*}M\left(\beta S,\mathcal{F}_{\left(S,\cdot\right)}\right)=\overline{K\left(\beta S,\odot\right)}. \end{equation*} \end{thm}

\begin{proof} From Theorem \ref{gene} we have that, for every ultrafilter $\U\in\beta S$, 
\begin{equation*} \mathcal{C}_{\mathcal{F}_{\left(S,\cdot\right)}}\left(\U\right)=\overline{\left\{\U\odot\V\mid\V\in\beta S\right\}}. \end{equation*}
Since $\left(S,\cdot\right)$ is commutative then $\left(S,\cdot\right)$ is contained in the (topological and algebraic) center of $\left(\beta S,\odot\right)$. Therefore the center of $\left(\beta S,\odot\right)$ is dense (in the Stone topology) in $\beta S$. Hence for every $\U\in\beta S$ we have that $\mathcal{C}_{\mathcal{F}_{\left(S,\cdot\right)}}(\U)$, being the closure of a right ideal, is a (closed) bilateral ideal of $\beta S$ (see e.g. Theorem 2.19, \cite{11}). Therefore, for every $\U\in\beta S$ we have that 

\begin{equation}\label{pizza} \overline{K\left(\beta S,\odot\right)}\subseteq \mathcal{C}_{\mathcal{F}_{\left(S,\cdot\right)}}\left(\U\right) \end{equation}

since $\overline{K\left(\beta S,\odot\right)}$ is the minimal closed bilateral ideal in $\beta S$. Now, if $\U\in K\left(\beta S,\odot\right)$ then $\left\{\U\odot\V\mid \V\in\beta S\right\}\subseteq K\left(\beta S,\odot\right)$, therefore

\begin{equation}\label{fichi} \forall \U\in K\left(\beta S,\odot\right) \ \mathcal{C}_{\mathcal{F}_{\left(S,\cdot\right)}}\left(\U\right)\subseteq \overline{K\left(\beta S,\odot\right)}.\end{equation}

Equation (\ref{pizza}) shows that $\overline{K\left(\beta S,\odot\right)}\subseteq\bigcap\limits_{\U\in\beta S} \mathcal{C}_{\mathcal{F}_{\left(S,\cdot\right)}}\left(\U\right)$, while equation (\ref{fichi}) shows that $\bigcap\limits_{\U\in\beta S} \mathcal{C}_{\mathcal{F}_{\left(S,\cdot\right)}}\left(\U\right)\subseteq \overline{K\left(\beta S,\odot\right)}$. From Corollary \ref{genevier} we know that $M\left(\beta S,\mathcal{F}_{\left(S,\cdot\right)}\right)\neq\emptyset$, and it is immediate to notice from the definitions that 

\begin{equation*} M\left(\beta S,\mathcal{F}_{\left(S,\cdot\right)}\right)=\bigcap\limits_{\U\in\beta S}\mathcal{C}_{\mathcal{F}_{\left(S,\cdot\right)}}\left(\U\right).\end{equation*}
Therefore $M\left(\beta S,\mathcal{F}_{\left(S,\cdot\right)}\right)=\overline{K\left(\beta S,\odot\right)}$.\end{proof}

Let us note, in particular, that Theorem \ref{braiati} entails also Corollary \ref{genevier}. Moreover, this allows for another characterization of the elements of $M\left(\beta S,\F_{\left(S,\cdot\right)}\right)$.

\begin{cor}\label{topolino} Let $S$ be a commutative semigroup with identity and let $\U\in\beta S$. Then we have the following two properties:
\begin{enumerate}[leftmargin=*,label=(\roman*),align=left ]
\item $\U\in M\left(\beta S,\F_{\left(S,\cdot\right)}\right)$ if and only if every set $A\in\U$ is piecewise syndetic;
\item for every piecewise syndetic set $A$ there exists a maximal ultrafilter $\U$ such that $A\in\U$.
\end{enumerate} \end{cor}
\begin{proof} From Theorem \ref{braiati} we have that $\U\in M\left(\beta S,\F_{\left(S,\cdot\right)}\right)$ iff $\U\in\overline{K\left(\beta S,\odot\right)}$. Then the results follow from the known facts (see e.g. \cite{11}, Theorem 4.40) that $\U\in\overline{K\left(\beta S,\odot\right)}$ iff $\forall A\in\U$ $A$ is piecewise syndetic and that for every piecewise syndetic set $A$ there exists $\U\in\overline{K\left(\beta S,\odot\right)}$ such that $A\in\U$. \end{proof}

We conclude this section by observing that, as an immediate consequence of Corollary \ref{topolino}, we have the following property of piecewise syndetic sets in commutative semigroups.

\begin{cor} Let $S$ be an infinite commutative semigroup with identity. Let $A\subseteq S$ be piecewise syndetic. Then for every infinite $B\subseteq S$ there exists $C\subseteq B$, $C$ infinite, such that $C\leq_{\F_{\left(S,\cdot\right)}} A$. \end{cor}

\begin{proof} Let $A, B$ be given and let $\U$ be a nonprincipal ultrafilter such that $B\in\U$. From Corollary \ref{topolino} we deduce that there exists a maximal ultrafilter $\V$ such that $A\in\V$. Let $D\in \U$ be such that $D\leq_{\F_{\left(S,\cdot\right)}} A$. If we set $C=D\cap B$ we have the thesis. \end{proof}

\section{Relationships between maximal sets and maximal ultrafilters}\label{uic}

In this section we prove certain relationships between maximal sets and maximal ultrafilters which are based on the notion of "partition regular family of sets". Let us recall the definition.
\begin{defn}\label{vulo} A family $P$ of subsets of a semigroup\footnote{Actually the definition, as well as all the results of this section, holds even if $S$ is only a set and not a semigroup.} $S$ is partition regular if for every finite partition $S=A_{1}\cup...\cup A_{n}$ there exists an index $i\leq n$ such that $A_{i}\in P$. $P$ is strongly partition regular if for every set $A\in P$ and for every finite partition $A=A_{1}\cup...\cup A_{n}$ there is an index $i\leq n$ such that $A_{i}\in P$.\end{defn}
These notions can be equivalenty formulated in terms of ultrafilters.

\begin{prop}\label{hu} Let $P\subseteq\Par(S)$. Then:
\begin{enumerate}[leftmargin=*,label=(\roman*),align=left ]
	\item $P$ is partition regular if and only if there exists $\U\in\beta S$ such that $\U\subseteq P$;
	\item $P$ is strongly partition regular if and only if $P$ is a union of ultrafilters, i.e. if for every $A\in P$ there exists $\U\in\beta S$ such that $A\in\U$ and $\U\subseteq P$.
\end{enumerate}
\end{prop}

See, e.g., \cite{11}, Theorem 3.11 for a proof of Proposition \ref{hu}.\par
In this section we do not suppose $\fe$ to be upward directed: in fact, we want to show that this property of $\fe$ can be derived by properties of the family $M(S,\F)$. Our main result in this section is Theorem \ref{ciurla}, which shows that for every semigroup $S$ and for every family of functions $\F\subseteq\mathfrak{F}\left(S^{n},S\right)$ the family $M(S,\F)$ is weakly partition regular if and only if it is strongly partition regular. To arrive to this result we need three lemmas.

\begin{lem}\label{kikka} Let $S$ be a semigroup, let $n\geq 1$ be a natural number and let $\F\subseteq\mathfrak{F}\left(S^{n},S\right)$. If $M(S,\F)$ is weakly partition regular then $M(\beta S,\F)\neq\emptyset$, and
\begin{equation*} M(\beta S,\F)=\left\{\U\in\bN\mid \U\subseteq M(S,\F)\right\}. \end{equation*}
In particular $\fe$ is upward directed.
\end{lem}

\begin{proof} Since $M(S,\F)$ is weakly partition regular there exists $\U\subseteq M(S,\F)$. Since all sets $A\in\U$ are maximal in $(S,\fe)$ we obviously have that $\U\in M(\beta S,\F)$, so $M(\beta S,\F)\neq\emptyset$ and $\fe$ is upward directed. Moreover, let $\V\in M(\beta S,\F)$: since $\U\fe\V$, for every set $B\in\V$ there exists $A\in\U$ such that $A\fe B$. $B$ has to be maximal since $A$ is maximal, hence $\V\subseteq M(S,\F)$.\end{proof}

A consequence of Lemma \ref{kikka} is the following:

\begin{lem} Let $S$ be a semigroup, let $n\geq 1$ be a natural number and let $\F\subseteq\mathfrak{F}\left(S^{n},S\right)$. If $M(S,\F)$ is strongly partition regular then 
\begin{equation*}M(S,\F)=\left\{A\subseteq S\mid \exists \U\in M(\beta S,\F) \ \text{such that} \ A\in\U\right\}.\end{equation*}\end{lem}

\begin{proof} Let $M(S,\F)$ be strongly partition regular. Let $A\in M(S,\F)$. From Proposition \ref{hu} we deduce that there exists $\U\subseteq M(S,\F)$ such that $A\in\U$. It is immediate to see that, since every set $B\in\U$ is maximal, then $\U\in M(\beta S,\F)$. Conversely, let $A$ be such that there exists $\U\in M(\beta S,\F) \ \text{such that} \ A\in\U$. Let $\V$ be an ultrafilter such that $B\in M(S,\F)$ for every $B\in\V$. Since $\U$ is maximal, $\V\fe\U$. Then there exists $B\in\V$ such that $B\fe A$ and, since $B$ is maximal, we deduce that $A$ is maximal. So $\U\subseteq M(S,\F)$.\end{proof}

Conversely, knowing that $(\beta S,\fe)$ is upward directed (i.e., that $M(\beta S,\F)\neq \emptyset$) gives informations about $M(S,\F)$ and $M(\beta S,\F)$:

\begin{lem}\label{lola} Let $S$ be a semigroup, let $n\geq 1$ be a natural number and let $\F\subseteq\mathfrak{F}\left(S^{n},S\right)$. Let us suppose that $M(\beta S,\F)\neq\emptyset$. Then $M(S,\F)\subseteq\bigcup M(\beta S,\F)$. \end{lem}

\begin{proof} Let $A\in M(S,\F)$, and let us suppose by contradiction that for every maximal ultrafilter $\U$ the set $A$ is not in $\U$, i.e. that the complement $A^{c}$ is in $\U$. Let $\U\in M(\beta S,\F)$ and let $\alpha\in\!^{\ast}\! \N$ be a generator of $\U$. In particular, $\alpha\in\!^{\ast}\! (A^{c})$.\par
Since $A\in M(S,\F)$, $A^{c}\leq_{\mathcal{F}} A$ so, by applying Proposition \ref{nonstandchar}, we know that there is a function $\varphi$ in $^{\ast}\! \mathcal{F}$ with $\varphi\left([A^{c}]^{n}\right)\subseteq \!^{\ast}\! A$. By transfer, this implies that 
\begin{equation*} ( ^{\ast}\! \varphi)\left(\:^{\ast}\! [A^{c}]^{n}\right)\subseteq ^{\ast\ast}\! A. \end{equation*}
As $\alpha\in\!^{\ast}\! A^{c}$, this entails that $(\:^{\ast}\! \varphi)(\alpha,\alpha,\dots,\alpha)\in\!^{\ast\ast}\! A$, so $A\in\mathfrak{U}_{( ^{\ast}\! \varphi)(\alpha,\dots,\alpha)}$. From Theorem \ref{questoqui} we deduce that $\mathfrak{U}_{\alpha}\fe \mathfrak{U}_{( ^{\ast}\! \varphi)(\alpha,\dots,\alpha)}$ and, since $\mathfrak{U}_{\alpha}$ is maximal, this entails that $\mathfrak{U}_{( ^{\ast}\! \varphi)(\alpha,\dots,\alpha)}$ is maximal. This is absurd, since this would entail that both $A$ and $A^{c}$ are in $\mathfrak{U}_{( ^{\ast}\! \varphi)(\alpha,\dots,\alpha)}$.\end{proof}

By combining the results proved in this section we obtain the following result, that we will use repeatedly in the next section.

\begin{thm}\label{ciurla} Let $S$ be a semigroup, let $n\geq 1$ be a natural number and let $\F\subseteq\mathfrak{F}\left(S^{n},S\right)$. Then the following two conditions are equivalent:
\begin{enumerate}[leftmargin=*,label=(\roman*),align=left ]
	\item\label{ciu1} $M(S,\F)$ is weakly partition regular;
	\item\label{ciu2} $M(S,\F)$ is strongly partition regular.
\end{enumerate}
\end{thm}

\begin{proof} $\ref{ciu1}\Rightarrow \ref{ciu2}$ Let us suppose that $M(S,\F)$ is weakly partition regular. Then from Lemma \ref{kikka} we deduce that that $\bigcup M(\beta S,\F)\subseteq M(S,\F)$ and that $M(\beta S,\F)\neq\emptyset$, so we can apply also Lemma \ref{lola} to obtain that $M(S,\F)\subseteq \bigcup M(\beta S,\F)$. So
\begin{equation*} M(S,\F)=\bigcup M(\beta S,\F) \end{equation*}
and we conclude by applying Proposition \ref{hu}.\par 
$\ref{ciu2}\Rightarrow \ref{ciu1}$ This implication holds for every family of sets.\end{proof}

\section{Applications}\label{appl}

In this section we want to show a few simple applications of $\F$-finite embeddabilities to combinatorial number theory. Our approach is based on the notion of $\fe$-upward closed family of sets.

\begin{defn} Let $S$ be a semigroup, let $n\geq 1$ be a natural number and let $\F\subseteq\mathfrak{F}(S^{n},S)$. We say that a family $P$ of subsets of $S$ is $\fe$-upward closed if $\forall A\in P$, $\forall B\in\Par(S)$, if $A\fe B$ then $B\in P$.\end{defn}

\begin{prop}\label{gorgo} Let $S$ be a semigroup, let $n\geq 1$ be a natural number and let $\F\subseteq\mathfrak{F}(S^{n},S)$. Let $P\neq\emptyset$ be a $\leq_{\mathcal{F}}$-upward closed family of subsets of $S$. Then:
\begin{enumerate}[leftmargin=*,label=(\roman*),align=left ]
	\item\label{uzi} $\forall \U,\V\in\beta S$, if $\U\subseteq P$ and $\U\fe\V$ then $\V\subseteq P$;
	\item\label{uzzi} if $M(\beta S,\F)\neq\emptyset$ and $P$ is weakly partition regular then $\U\subseteq P$ for every $\U\in M(\beta S,\F)$;
	\item\label{uzzzi} if $n=1$ and $\F=\F^{R}_{(S,\cdot)}$ (resp., $\F=\F^{L}_{(S,\cdot)}$) then 
	\begin{equation*} I_{P}=\{\U\in\beta S\mid \U\subseteq P\}\end{equation*}
	is a right (resp. left) closed ideal in $(\beta S,\odot)$. In particular, if $S$ is commutative then $I_{P}$ is a closed bilateral ideal in $(\beta S,\odot)$, hence $\overline{K(\beta S,\odot)}\subseteq I_{P}$.
\end{enumerate}
  \end{prop}
	 
\begin{proof} \ref{uzi} Let $\U\subseteq P$ and let $\U\fe\V$. Let $A\in\V$ and let $B\in\U$ be such that $B\fe A$. $B\in P$, therefore $A\in P$, so $\V\subseteq P$.\par
\ref{uzzi} Let $\V$ be an ultrafilter such that $A\in P$ for every $A\in\V$ and let $\U\in M(\beta S,\F)$. Let $A\in\U$. Since $\U\in M(\beta S,\F)$ there exists $B\in\V$ such that $B\fe A$. By construction, $B\in P$, so $A\in P$.\par
\ref{uzzzi} If $I_{P}$ is empty then the thesis holds. Otherwise, let $\U\in I_{P}$. From Proposition \ref{CommUpDir} we know that $\U\leq_{\F^{R}_{(S,\cdot)}}\U\odot\V$, therefore from \ref{uzi} we deduce that $\U\odot\V\in I_{P}$. Hence $I_{P}$ is a right ideal in $(\beta S,\odot)$. It is routine to prove that $I_{P}$ is also closed.\end{proof}

{\bfseries Example:} It is immediate to prove that the family
\begin{multline*} TGAP:=\Big\{A\subseteq\N\mid A \ \text{contains a translate of}\\
\text{arbitrarily long geoarithmetic progressions}\Big\}\end{multline*}
is $\leq_{\F_{(\N,+)}}$-upward invariant. It is also partition regular (since it extends the family of GAP-rich sets, which is partition regular as proved by V. Bergelson in \cite{3}). Therefore from Proposition \ref{gorgo} we deduce that for every ultrafilter $\U\in\overline{K(\bN,\oplus)}$, for every $A\in\U$, $A\in TGAP$.\\\par
 
Results similar to Proposition \ref{gorgo} (framed for finite embeddabilities) have been used in \cite{4}, \cite{14}, \cite{16} to reprove some known results regarding piecewise syndetic sets in $(\N,+)$ and ultrafilters in $\overline{K(\bN,\oplus)}$. In the next sections we will use Proposition \ref{gorgo} to prove some results regarding other important families of sets and ultrafilters.

\subsection{Arithmetic progressions}\label{fruttj}
Let $S=\N$ and let $\F=\mathcal{A}$.\footnote{We recall that $\mathcal{A}=\left\{f_{a,b}\in\mathfrak{F}(\N,\N)\mid a,b\in\N, a\neq 0 \ \text{and} \ f_{a,b}(n)=an+b \ \forall n\in\N\right\}$.} We noticed in Section \ref{cui} that $\A$ is transitive, reflexive and upward directed, and we observed in Section \ref{ciu} that, by applying Proposition \ref{maxset}, we have that
\begin{equation*} M(\N,\A)=\left\{A\subseteq\N\mid A \ \text{is AP-rich}\right\}. \end{equation*}
The original version of Van der Waerden Theorem (see \cite{23}) on arithmetic progressions states that the family of AP-rich subsets of $\N$ is partition regular. Since $\leq_{\A}$ is upward directed, from Theorem \ref{ciurla} we obtain a new proof of the (known) stronger version of Van der Waerden Theorem:

\begin{thm}\label{vdws} The family $AP=\left\{A\subseteq\N\mid A \ \text{is AP-rich} \right\}$ is strongly partition regular.\end{thm} 

In particular from Proposition \ref{kikka} we have that $M(\N,\A)=\bigcup M(\bN,\A)$, so we obtain that
\begin{equation*} M(\bN,\A)=\left\{\U\in\bN\mid\forall A\in\U \ A\ \text{is AP-rich}\right\}.\end{equation*}
Moreover, by definition we have that $\F_{(\N,+)}\subseteq\A$ and $\F_{(\N,\cdot)}\subseteq\A$, so from Proposition \ref{papavero} it follows that
\begin{equation}\label{incl} \overline{K(\bN,\oplus)}\cup\overline{K(\bN,\odot)}\subseteq M(\bN,\A). \end{equation}
This result can be improved by noticing that $AP$ is both $\mathcal{F}_{(\N,+)}$- and $\mathcal{F}_{(\N,\cdot)}$-upward invariant, therefore $M(\bN,\A)$ is a bilateral ideal both in $(\bN,\oplus)$ and in $(\bN,\odot)$. It is not difficult to prove that the reverse inclusion of (\ref{incl}) does not hold.

\begin{prop} There exists $\U\in M(\bN,\A)$ such that $\U\notin\overline{K(\bN,\oplus)}\cup\overline{K(\bN,\odot)}$.\end{prop}

\begin{proof} From Tao-Green Theorem (see \cite{10}) it is known that the set $P$ of prime numbers contains arbitrarily long arithmetic progressions. Let $\U\in M(\bN,\A)$ be an ultrafilter such that $P\in\U$. We claim that $\U\notin\overline{K(\bN,\oplus)}\cup\overline{K(\bN,\odot)}$.\par
$\U\notin\overline{K(\bN,\oplus)}$: $P$ is not piecewyse syndetic in $(\N,+)$.\par
$\U\notin\overline{K(\bN,\odot)}$: let us consider the family 
\begin{equation*} S=\left\{A\subseteq\N\mid A \ \text{contains at least two even numbers}\right\}.\end{equation*}
It is immediate to see that $S$ is weakly partition regular and $\leq_{\F_{(\N,\cdot)}}$-upward closed. Moreover, from Theorem \ref{braiati} we have that $M(\bN,\F_{(\N,\cdot)})=\overline{K(\bN,\odot)}$ so, for every $\V\in\overline{K(\bN,\odot)}$, from Proposition \ref{gorgo} it follows that $\V\subseteq S$. Since $P\notin S$, we can conclude that $\U\notin\overline{K(\bN,\odot)}$.\end{proof}

Theorem \ref{gene} can be used to characterize the cones $\mathcal{C}_{\mathcal{A}}(\U)$. It is easy to see that $\mathcal{A}$ is the set of functions generated by $G:\N\times\N^{2}\rightarrow\N$, where
\begin{equation*} \forall n\in\N, \forall (a,b)\in \N^{2} \ G(n,(a,b))=an+b, \end{equation*}
with set of parameters $R=\N^{2}\setminus \left\{(0,n)\mid n\in\N\right\}$. From Theorem \ref{gene} it follows that, for every ultrafilter $\U$ in $\bN$, 
\begin{equation*} \mathcal{C}_{\A}(\U)=\overline{\left\{\overline{G}(\U\otimes\V)\mid \V\in\Theta_{R}\right\}}. \end{equation*}
The previous characterization can be made more explicit by recalling that
\begin{equation*} A\in\overline{G}(\U\otimes\V) \ \text{if and only if} \ \left\{n\in\N\mid \left\{(a,b)\in\N^{2}\mid an+b\in A\right\}\in\V\right\}\in\U. \end{equation*}
Finally, this characterization of the cones, together with the fact that $\overline{K(\bN,\oplus)}\subseteq M(\bN,\A)$, provides another characterization of AP-rich sets and maximal ultrafilters, namely:
\begin{itemize}
	\item $A\subseteq\N$ is AP-rich iff for every ultrafilter $\U\in\overline{K(\bN,\oplus)}$ there exists $\V\in\Theta_{R}$ such that $A\in\overline{G}(\U\otimes\V)$;
	\item $M(\bN,\A)=\overline{\left\{\overline{G}(\U\otimes\V)\mid \V\in\Theta_{R}\right\}}$ for every $\U\in \overline{K(\bN,\oplus)}\cup\overline{K(\bN,\odot)}$.
\end{itemize}

\subsection{Generalized arithmetic progressions}
The first generalization of arithmetic progressions that we consider are the polynomial progressions introduced by R. Hirschfeld in \cite{hirsch}. Let us recall their definition:
\begin{defn} A polynomial progression of length $l$ and degree $d$ is a sequence of the form $\left\{P(1),P(2),\dots,P(l)\right\}$ where $P(x)$ is a polynomial with natural coefficients whose degree is $d$. In this case we say that the sequence $\left\{P(1),P(2),\dots,P(l)\right\}$ is generated by $P(x)$. Moreover, given a semigroup $S\subseteq\N$ and a subset $D\subseteq\{0,\dots,d\}$ we say that $P(x)=\sum\limits_{i=0}^{d} a_{i} x^{i}$ is a $(S,D)$-polynomial if, for every $i\leq d$, $a_{i}\in S$ if $i\in D$ and $a_{i}=0$ if $i\notin D$. \end{defn}

The main result of \cite{hirsch} can be restated as follows. 

\begin{thm}\label{rufi} For every finite partition $\N=A_{1}\cup\dots\cup A_{n}$ there is an index $i\leq n$ such that for arbitrary large $k\in\N$ there are polynomial progressions of length $k$ and degree $d$ in $A_{i}$ generated by a $(S,D)$-polynomial.
\end{thm} 

Given a semigroup $S\subseteq \N$ and $D\subseteq\{0,\dots,d\}$, let
\begin{equation*} \mathcal{P}_{S,D}=\{P(x)\mid P(x) \ \text{is a} \ (S,D)\text{-polynomial}\}.\end{equation*} 
From Proposition \ref{maxset} it is immediate to see that 
\begin{multline*} M(\N,\mathcal{P}_{S,D})=\{A\subseteq \N\mid A \ \text{contains arbitrarily long polynomial}\\\text{progressions of degree}\ d \ \text{generated by a} \ (S,D)\text{-polynomial}\}.\end{multline*}
Therefore Theorem \ref{rufi} is equivalent to say that $M(\N,\mathcal{P}_{S,D})$ is partition regular. Hence we can apply Theorem \ref{ciurla} and we obtain the following result.

\begin{thm}\label{accontentiamoci} The family $M(\N,\mathcal{P}_{S,D})$ is strongly partition regular.\end{thm}

Let us note that Theorem \ref{vdws} is a particular case of Theorem \ref{accontentiamoci}. Now let $D=\{d_{0},\dots,d_{k}\}\subseteq\{0,\dots,k\}$. By definition, $\mathcal{P}_{S,D}$ is generated by the pair $(G,S)$, where $\forall x\in\N, \forall a_{d_{0}},\dots,a_{d_{k}}\in S$ we have
\begin{equation*} G(x,a_{d_{0}},\dots,a_{d_{k}})=\sum\limits_{i=0}^{k}a_{d_{i}}x^{d_{i}}.\end{equation*}
From Bonus 3 in \cite{hirsch} we know that $\overline{K(\bN,\oplus)}\subseteq M(\bN,\mathcal{P}_{S,D})$, so from Theorem \ref{gene} we obtain the following characterization of $M(\bN, \mathcal{P}_{S,D})$: given any ultrafilter $\U\in\overline{K(\bN,\oplus)}$, we have that
\begin{equation*} M(\bN, \mathcal{P}_{S,D})=\overline{\{\overline{G}(\U\otimes\V)\mid \V\in\Theta_{S}\}}.\end{equation*}
Now let us consider geoarithmetic progressions. In \cite{3}, V. Bergelson proved (as a consequence of a more general results regarding multiplicatively large sets) the following result regarding geoarithmetic progressions. 

\begin{thm}[V. Bergelson, Theorem 1.4 in \cite{3}]\label{bergelson} Let $n,r\in\N$ and let $\N=A_{1}\cup\dots\cup A_{r}$. Then there exists $k\leq r$, $a,b\in\N$ and $d,q\in A_{k}$ such that
\begin{equation*} \left\{bq^{j}(a+id)\mid 0\leq i,j\leq n\right\}\subseteq A_{k}. \end{equation*}\end{thm}

As a consequence of Theorem \ref{bergelson} we have, in particular, that the family of GAP-rich sets is partition regular. This result was improved in various ways in \cite{berg}, obtaining Ramsey-theoretical results related to geoarithmetic progressions for semigroups, as well as some algebraical properties of the set of ultrafilters whose elements are GAP-rich sets. Here we want to prove a few more results about these ultrafilters.

\begin{prop} The family $\left\{A\subseteq \N\mid A \ \text{is GAP-rich}\right\}$ is $\leq_{\F_{(\N,\cdot)}}$-upward invariant.\end{prop}

\begin{proof} Let $A$ be GAP-rich. Let $A\leq_{\F_{(\N,\cdot)}} B$. Let $n\in\N$ and let $a,b\in\N$ and $d,q\in A$ such that $\left\{bq^{j}(a+id)\mid 0\leq i,j\leq n\right\}\subseteq A$. Let $m\in\N$ be such that $m\cdot \left\{bq^{j}(a+id)\mid 0\leq i,j\leq n\right\}\subseteq B$. Then $\left\{mbq^{j}(a+id)\mid 0\leq i,j\leq n\right\}\subseteq B$, therefore $B$ is GAP-rich.\end{proof}

Hence, from Proposition \ref{gorgo} we obtain the following (known, see e.g. \cite{berg}, Corollary 4.5) result regarding $\overline{K(\bN,\odot)}$.

\begin{cor}\label{jezz} The family of ultrafilters whose elements are GAP-rich is a closed bilateral ideal in $\overline{K(\bN,\odot)}$. Therefore for every ultrafilter $\U\in\overline{K(\bN,\odot)}$, for every set $A\in\U$, $A$ is GAP-rich.\end{cor}

As in Section \ref{ciu}, let 
\begin{multline*} \mathcal{G}=\Big\{f_{r,a,b}\left(n,m\right)\in\mathfrak{F}(\N^{2},\N)\mid r,a,b\in\N, r>1, b>0 \ \text{and} \\
f_{r,a,b}\left(n,m\right)=r^{n}\left(a+mb\right) \ \forall \left(n,m\right)\in\N^{2}\Big\}. \end{multline*}	From Proposition \ref{maxset} it is immediate to notice that the family of GAP-rich sets can be characterized as the family of maximal sets for $\leq_{\mathcal{G}}$-finite embeddability. Therefore from Theorem \ref{ciurla} we get the following result.

\begin{prop} The family of GAP-rich sets is strongly partition regular. \end{prop}

Finally, by applying Theorem \ref{gene} we obtain a characterization of $M\left(\bN,\mathcal{G}\right)$. In fact, as we noticed in Section \ref{trevisan}, $\mathcal{G}$ is generated by the pair $(G,R)$, where 
\begin{equation*} G\left(\left(n,m\right),\left(r,a,b\right)\right)=r^{n}\left(a+mb\right), \ R=\left\{\left(r,a,b\right)\in\N^{3}\mid r>1, b>0\right\}.\end{equation*}
Therefore, from Theorem \ref{gene} we obtain that the family of ultrafilters $\U$ such that every set $A\in\U$ is GAP-rich is
\begin{equation*}\label{garit} M\left(\bN,\mathcal{G}\right)=\overline{\left\{\overline{G}\left(\W\otimes\V\right)\mid \W\in \mathfrak{G}\left(\U^{\left(2\right)}\right), \V\in\Theta_{R}\right\}},\end{equation*}
where $\U$ is any maximal ultrafilter in $(\bN,\mathcal{G})$. For example from Corollary \ref{jezz} we can take any $\U\in\overline{K(\bN,\odot)}$.

\subsection{Partition regularity of diophantine equations on $\N$}
An interesting topic in combinatorial number theory is the study of the partition regularity of nonlinear diophantine equations\footnote{Let us recall that, given a polynomial $P(x_{1},\cdots,x_{n})$ with integer coefficient, the equation $P(x_{1},\cdots,x_{n})=0$ is partition regular if and only if for every finite partition $\N=A_{1}\cup\dots\cup A_{m}$ there is an index $i\leq m$ and elements $a_{1},\dots,a_{n}\in A_{i}\setminus\{0\}$ such that $P(a_{1},\dots,a_{n})=0$.} (see e.g. \cite{15}, \cite{17}). In this section we want to prove a result regarding homogeneous equations. We recall that one of the most interesting (and challenging) open problems in this field (posed in 1975 by P. Erd\"{o}s and R. Graham) regards a homogeneous equation: in fact, it concerns the partition regularity of the pythagorean equation
\begin{equation*} x^{2}+y^{2}-z^{2}=0. \end{equation*} 
The result that we want to prove concerns a relation between homogenous partition regular equations and $\leq_{\F_{(\N,\cdot)}}$. Let $P(x_{1},\cdots,x_{n})$ be a homogeneous polynomial with integer coefficients\footnote{Throughout this section, all the polynomials we consider will have integer coefficients.}, and let $R_{P}$ be the set
\begin{equation*} R_{P}=\left\{A\subseteq\N\mid \exists a_{1},...,a_{n}\in A \ \text{such that} \ P(a_{1},...,a_{n})=0\right\}. \end{equation*}

\begin{prop}\label{paperina} The following facts are equivalent:
\begin{enumerate}[leftmargin=*,label=(\roman*),align=left ]
\item\label{fy1} $R_{P}$ is partition regular;
\item\label{fy2} for every $\U\in\overline{K(\bN,\odot)}$ we have that $\U\subseteq R_{P}$.
\end{enumerate}
Moreover, if $R_{P}$ is partition regular then
\begin{equation*} I_{P}=\{\U\in\bN\mid \forall A\in\U \ A\in R_{P}\} \end{equation*}
is a closed bilateral ideal in $(\bN,\odot)$.\end{prop}

\begin{proof} Since $P(x_{1},\cdots,x_{n})$ is homogeneous then $R_{P}$ is clearly $\leq_{\F_{(\N,\cdot)}}$-upward invariant, so that $\ref{fy1}\Rightarrow \ref{fy2}$ follows from Proposition \ref{gorgo}. The converse is an immediate consequence of Proposition \ref{hu}. Finally, that $I_{P}$ is a bilateral ideal whenever $R_{P}$ is partition regular is again a consequence of Proposition \ref{gorgo}.\end{proof}

\begin{cor} The set 
\begin{multline*} H=\Big\{\U\in\bN\mid \text{for every homogeneous polynomial} \ P(x_{1},\dots,x_{n})\\
\text{if} \ R_{P} \ \text{is partition regular then}\ \U\subseteq R_{P}\Big\}\end{multline*}
is a closed bilateral ideal in $(\bN,\odot)$. In particular, $\overline{K(\bN,\odot)}\subseteq H$. \end{cor}
\begin{proof} It is immediate to notice that, if 
\begin{equation*} P_{H}=\{P(x_{1},\dots,x_{n})\mid R_{P} \ \text{is partition regular}\}\end{equation*}
then 
\begin{equation*} H=\bigcap\limits_{P\in P_{H}} I_{P}. \end{equation*}
Therefore $H$ is an intersection of closed bilateral ideals, hence it is a closed bilateral ideal. \end{proof}
Since $\U\in\overline{K(\bN,\odot)}$ iff $\forall A\in\U$ $A$ is piecewise syndetic in $(\N,\cdot)$, we conclude with the following characterization of homogeneous partition regular polynomial. 

\begin{cor}\label{pitagora} Let $P(x_{1},\cdots,x_{n})$ be a homogeneous polynomial. The following facts are equivalent: 

\begin{enumerate}[leftmargin=*,label=(\roman*),align=left ]
\item $P(x_{1},\cdots,x_{n})$ is partition regular;
\item for every piecewise syndetic set $A$ in $(\N,\cdot)$ there exists $a_{1},\dots,a_{n}\in A$ such that $P(a_{1},\dots, a_{n})=0$.
\end{enumerate}\end{cor}


\begin{thebibliography}{}

\bibitem[1]{berg} M. Beiglb\"{o}ck, V. Bergelson, N. Hindman, D. Strauss, \emph{Some new results in multiplicative and additive Ramsey theory}, Trans. Amer. Math. Soc. 360 (2008), 819--847.

\bibitem[2]{1} M. Beiglb\"{o}ck, \emph{An ultrafilter approach to Jin's theorem,} Israel J. Math. 185 (2011), 369--374.

\bibitem[3]{2} V. Benci, M. Di Nasso, \emph{Alpha-Theory: an elementary axiomatics for nonstandard analysis}, Expo. Math. 21 (2003), 355--386.

\bibitem[4]{3} V. Bergelson, \emph{Multiplicatively large sets and ergodic Ramsey theory}, Israel J. Math. 148 (2005), 23--40.

\bibitem[5]{4} A. Blass, M. Di Nasso, \emph{Finite Embeddability of sets and ultrafilters}, submitted, arXiv: 1405.2841. 

\bibitem[6]{6} C. C. Chang, H. J. Keisler, {\bfseries Model theory} (3rd ed.), North-Holland, Amsterdam, 1990.

\bibitem[7]{7} M. Davis, {\bfseries Applied Nonstandard Analysis}, John Wiley \& Sons (1977).

\bibitem[8]{8} M. Di Nasso, \emph{Embeddability properties of difference sets}, Integers 14 (2014), A27.

\bibitem[9]{9} M. Di Nasso, \emph{Iterated hyper-extensions and an idempotent ultrafilter proof of Rado's theorem}, Proc. Amer. Math. Soc. 143 (2015), 1749--1761. 

\bibitem[10]{dinasso} M. Di Nasso, \emph{A taste of nonstandard methods in combinatorics of numbers}, in "Geometry, Structure and Randomness in Combinatorics", CRM Series, Scuola Normale Superiore, Pisa (2015), 27--46.

\bibitem[11]{10} B. Green, T. Tao, \emph{The primes contain arbitrarily long arithmetic progressions}, Ann. Math. 167 (2008), 481--547.
   
\bibitem[12]{11} N. Hindman, D. Strauss, {\bfseries Algebra in the Stone-\v{C}ech Compactification} (2nd edition), de Gruyter (2012).

\bibitem[13]{12} N. Hindman, D. Strauss, \emph{Density in arbitrary semigroups}, Semigroup Forum 73 (2006), 273--300.

\bibitem[14]{hirsch} R. Hirschfeld, \emph{On a generalization of the van der Waerden Theorem}, European J. Combin. 30 (2009), 617--621.

\bibitem[15]{13} P. Krautzberger, \emph{Idempotent Filters and Ultrafilters}, Ph.D. thesis, Freie Univ. Berlin (2009).

\bibitem[16]{14} L. Luperi Baglini, \emph{Hyperintegers and Nonstandard Techniques in Combinatorics of Numbers}, PhD thesis, University of Siena (2012), arXiv 1212.2049.

\bibitem[17]{15} L. Luperi Baglini, \emph{Partition Regularity of Nonlinear Polynomials: a Nonstandard Approach}, Integers 14 (2014), A30.

\bibitem[18]{16} L. Luperi Baglini, \emph{Ultrafilters maximal for finite embeddability}, J. Log. Anal. 6 (2014), A6.

\bibitem[19]{17} L. Luperi Baglini, \emph{A nonstandard technique in combinatorial number theory}, European J. Combin., in press, DOI 10.1016/j.ejc.2015.02.010.
 
\bibitem[20]{18} C. Puritz, \emph{Ultrafilters and standard functions in nonstandard analysis}, Proc. London Math. Soc. 22 (1971), 706--733.

\bibitem[21]{19} C. Puritz, \emph{Skies, Constellations and Monads}, in {\bfseries Contributions to Non-Standard Analysis}, (WAJ Luxemburg and A. Robinson eds), North Holland 1972, 215--243.

\bibitem[22]{20} R. Rado, \emph{Studien zur Kombinatorik}, Math. Z. 36 (1933), 242--280.

\bibitem[23]{22} I. Z. Ruzsa, \emph{On difference sets}, Studia Sci. Math. Hungar. 13 (1978), 319--326.

\bibitem[24]{23} B. van der Waerden, \emph{Beweis einer Baudetschen Vermutung}, Nieuw Arch. Wiskunde 19 (1927), 212--216.

 



\end{thebibliography}
\end{document}